\documentclass{amsart}
\usepackage[utf8]{inputenc}
\usepackage{lmodern}% http://ctan.org/pkg/lm
\usepackage{mathrsfs} 
\usepackage{amssymb}
\usepackage{bm}
\usepackage{graphicx}
\usepackage[centertags]{amsmath}
\usepackage{amsfonts}
\usepackage{amsthm}
\usepackage{enumerate}
\usepackage{tocvsec2}
\usepackage{xcolor}
\usepackage[margin=1in]{geometry}

\newtheorem{theorem}{Theorem}[section]
\newtheorem*{theorem*}{Theorem}

\newtheorem{corollary}[theorem]{Corollary}
\newtheorem{lemma}[theorem]{Lemma}
\newtheorem{rem}[theorem]{Remark}

\newtheorem{proposition}[theorem]{Proposition}

\newtheorem{fact}[theorem]{Fact}

\theoremstyle{definition}

\newcommand{\ee}{\varepsilon}
\newcommand{\nn}{\mathbb{N}}
\newcommand{\rr}{\mathbb{R}}

\begin{document}

\title{Factorization of Asplund operators}

\begin{abstract} We give necessary and sufficient conditions for an operator $A:X\to Y$ on a Banach space having a shrinking FDD to factor through a Banach space $Z$ such that the Szlenk index of $Z$ is equal to the Szlenk index of $A$. We also prove that for every ordinal $\xi\in (0, \omega_1)\setminus\{\omega^\eta: \eta<\omega_1\text{\ a limit ordinal}\}$, there exists a Banach space $\mathfrak{G}_\xi$ having a shrinking basis and Szlenk index $\omega^\xi$ such that for any separable Banach space $X$ and any operator $A:X\to Y$ having Szlenk index less than $\omega^\xi$, $A$ factors through a subspace and through a quotient of $\mathfrak{G}_\xi$, and if $X$ has a shrinking FDD, $A$ factors through $\mathfrak{G}_\xi$.

\end{abstract}

\author{R.M. Causey}
\address{Department of Mathematics, Miami University, Oxford, OH 45056, USA}
\email{causeyrm@miamioh.edu}

\author{K.V. Navoyan}
\address{Department of Mathematics, University of Mississippi, Oxford, MS 38655, USA}
\email{KNavoyan@go.olemiss.edu}


\maketitle

\section{Introduction}

A celebrated result in Banach space theory is the factorization theorem of Davis, Figiel, Johnson, and Pe\l czy\'{n}ski \cite{DFJP}, which states that any weakly compact operator factors through a reflexive Banach space.  Since then, a number of classes of operators have been shown to be characterized by such  factorization property.  Beauzamy \cite{Beau} showed that any Rosenthal operator (that is, any operator not preserving an isomorphic copy of $\ell_1$) factors through a Banach space which contains no isomorphic copy of $\ell_1$, and Re\u{\i}nov \cite{Rei}, Heinrich \cite{Hei}, and Stegall \cite{Ste} independently showed that any Asplund operator factors through an Asplund Banach space. In contrast, Beauzamy \cite{Beau1} showed that there exist super weakly compact operators which do not factor through any superreflexive Banach space.   This turns out to be a particular case of a quantitative factorization problem.   More precisely, there exists an ordinal index, called the \emph{James index}, denoted by $\mathcal{J}$, which takes an operator and returns an ordinal if that operator is weakly compact, and (by convention) returns the symbol $\infty$ if that operator is not weakly compact.     It was shown in \cite{qf} that if an operator $A$ satisfies $\mathcal{J}(A)\leqslant \omega^{\omega^\xi}$, then $A$ factors through a Banach space $Z$ with $\mathcal{J}(I_Z)\leqslant \omega^{\omega^{\xi+1}}$, which is a quantified version of the David, Figiel, Johnson, Pe\l czy\'{n}ski factorization result. Given that the class of super weakly compact operators is precisely the class of operators whose James index does not exceed $\omega$, Beauzamy's negative factorization result from \cite{Beau1} witnesses that the passage from the upper estimate $\mathcal{J}(A)\leqslant \omega^{\omega^\xi}$ to a strictly larger upper estimate on the James index of $I_Z$ for a space through which $A$ factors is necessary.   Similarly, Brooker \cite{Brooker} showed that an operator with Szlenk index $\omega^\xi$ always factors through a Banach space with Szlenk index not more than $\omega^{\xi+1}$, and that there exist certain ordinals $\xi$ and operators for which the passage from $\omega^\xi$ to $\omega^{\xi+1}$ is optimal.   Such quantified factorization theorems yield useful information regarding universal factorization spaces (see \cite{qf} for further information), generalizing Figiel's example \cite{Figiel} of a separable, reflexive Banach space $Z$ such that every compact operator factors through a subspace of $Z$. 

The main goal of this work is to extend Brooker's result regarding factorization of Asplund operators.  That is, if $A:X\to Y$ is an operator and $Sz(A)=\omega^\xi$, one would like to know when $A$ factors through a Banach space $Z$ such that $Sz(Z)=\omega^\xi$, or when the passage to $\omega^{\xi+1}$ is optimal. We completely solve this problem in the case that the domain of the operator has a shrinking basis. All relevant notions regarding the $\ee$-Szlenk indices will be defined later. This result extends the factorization result of Kutzarova and Prus, which is the $\xi=1$ case of the following theorem.

\begin{theorem} Fix an ordinal $0<\xi<\omega_1$.  Suppose that $X$ is a Banach space with a shrinking basis and $A:X\to Y$ is an operator with $Sz(A)=\omega^\xi$. \begin{enumerate}[(i)]\item If $\xi=\omega^\zeta$, where $\zeta$ is a limit ordinal, then $A$ does not factor through any Banach space $Z$ with $Sz(Z)=Sz(A)$. \item $\xi=1$ or $\xi=\omega^{\zeta+1}$ for some ordinal $\zeta$, then $A$ factors through a Banach space $Z$ with $Sz(Z)=Sz(A)$ if and only if there exists an ordinal $\gamma<\omega^\xi$ such that for all $n\in\nn$, $Sz(A, 1/2^n)\leqslant \gamma^n$.   \item If $\xi=\beta+\gamma$ for some $\beta, \gamma<\xi$, then $A$ factors through a Banach space $Z$ with $Sz(Z)=Sz(A)$.  \end{enumerate}

\label{main1}

\end{theorem}

It follows from standard facts about ordinals that the cases listed in Theorem \ref{main1} are exhaustive.

The second major result of the paper is to establish an optimal result regarding universal Asplund spaces. 

\begin{theorem} Fix an ordinal $\xi\in (0,\omega_1)\setminus\{\omega^\eta: \eta\text{\ a limit ordinal}\}$.   Then there exists a Banach space $\mathfrak{G}_\xi$ with shrinking basis and $Sz(\mathfrak{G}_\xi)=\omega^\xi$ such that if $A:X\to Y$ is a separable range operator with $Sz(A)<\omega^\xi$, then $A$ factors through both a subspace and a quotient of $\mathfrak{G}_\xi$.

\end{theorem}

\section{Coordinate systems}

Throughout, $\mathbb{K}$ will denote the scalar field, which is either $\mathbb{R}$ or $\mathbb{C}$.   Given a subset $S$ of a Banach space, $[S]$ will denote the closed span of $S$. Given a Banach space $X$ and $K\subset X^*$ weak$^*$-compact, we let $r_K\equiv 0$ if $K=\varnothing$, and otherwise we let $r_K(x)=\max_{x^*\in K}\text{Re\ }x^*(x)$.

We recall that a \emph{Markushevich basis} (or $M$-\emph{basis}) for a Banach space is a biorthogonal system $(x_i, x^*_i)_{i\in I}\subset X\times X^*$ such that $[x_i:i\in I]=X$ and $\cap_{i\in I}\ker(x^*_i)=\{0\}$.    For us, an \emph{FMD} for the Banach space $X$ will be a sequence $\textsf{F}=(F_n)_{n=1}^\infty$ of subspaces of $X$ such that there exist an $M$-basis $(x_i, x^*_i)_{i=1}^\infty$ and a sequence $0=k_0<k_1<\ldots$ of natural numbers such that for each $n\in\nn$, $$F_n=[x_i: k_{n-1}<i\leqslant k_n].$$   If $K\subset X^*$, we say that an FMD $\textsf{F}$ for $X$ is $K$-\emph{shrinking} provided that there exist an $M$-basis $(x_i, x^*_i)_{i=1}^\infty$ and $0=k_0<k_1<\ldots$ such that $$F_n=[x_i: k_{n-1}<i\leqslant k_n]$$ and such that $K\subset [x^*_i:i\in\nn]$.    In the case that $A:X\to Y$ is an operator and $K=A^*B_{Y^*}$, we will say $\textsf{F}$ is $A$-\emph{shrinking} rather than $K$-shrinking.  If $K=B_{X^*}$, we will simply say $\textsf{F}$ is \emph{shrinking}.

We say a sequence $(u_i)_{i=1}^\infty$ in $X$ is a \emph{block sequence with respect to} $\textsf{F}$ provided that there exist natural numbers $0=k_0<k_1<\ldots$ such that for each $i\in\nn$, $u_i\in [F_j: k_{i-1}<j\leqslant k_i]$.

We will be primarily concerned with separable Banach spaces, and exclusively concerned with weak$^*$-fragmentable sets.    We recall that if $X$ is a Banach space and $K\subset X^*$ is weak$^*$-compact, we say $K$ is \emph{weak}$^*$-\emph{fragmentable} if for any $\ee>0$ and any non-empty subset $L$ of $K$, there exists a weak$^*$-open set $v\subset X^*$ such that $v\cap L\neq \varnothing$ and $\text{diam}(v\cap L)\leqslant \ee$.    It is a consequence of the Baire category theorem and topological considerations that if $X$ is a separable Banach space and $K\subset X^*$ is weak$^*$-compact, then $K$ is weak$^*$-fragmentable if and only if it is norm separable.

One benefit of the notion of a $K$-shrinking FMD is that if $X$ is separable and $K\subset X^*$ is norm separable, then $X$ admits a $K$-shrinking FMD. Indeed, assume $K$ is norm separable and does not lie in the span of finitely many vectors (otherwise the result is trivial).  We may fix $(v_n)_{n=1}^\infty \subset X$ norm dense in $X$, $(v^*_n)_{n=1}^\infty$ weak$^*$-dense in $X^*$, and $(u^*_n)_{n=1}^\infty$ norm dense in $K$.  By the usual method of constructing an $M$-basis for a separable Banach space, one recursively selects an $M$-basis $(x_n, x^*_n)_{n=1}^\infty$ having the property that for each $n\in\nn$, $v_n\in [x_i: i\leqslant 3n]$, $v^*_n, u^*_n\in [x^*_i: i\leqslant 3n]$.  The weakening of the notion of shrinking FMD to the notion of a $K$-shrinking FMD allows us to study norm separable subsets of the duals of Banach spaces with non-separable duals, for example $K=A^*B_{Y^*}$, where $A:\ell_1\to Y$ is an Asplund operator.  

The primary property of a $K$-shrinking FMD, say $\textsf{F}$,  with which we will be concerned is that a bounded block sequence $(y_n)_{n=1}^\infty$ with respect to $\textsf{F}$ must be $\sigma(X, K)$-null.  Indeed, suppose $(x_i, x^*_i)_{i=1}^\infty$ is an $M$-basis such that $K\subset [x^*_n: n\in\nn]$ and $0=k_0<k_1<\ldots$ is such that for each $n\in\nn$, $F_n=[x_i: k_{n-1}<i\leqslant k_n]$.    Then if $(y_n)_{n=1}^\infty$ is a bounded block sequence with respect to $\textsf{F}$, to see that $(y_n)_{n=1}^\infty$ is $\sigma(X ,K)$-null,  it is sufficient to know that $(y_n)_{n=1}^\infty$ is pointwise null on a subset of $X^*$ the closed span of which contains $K$.  We then note that $\{x^*_n: n\in\nn\}$ is such a set.

Given an FMD $\textsf{F}$ for the Banach space $X$, a weak$^*$-compact subset $K\subset X^*$, and an infinite subset $M$ of $\nn$,  we define a seminorm $\langle \cdot\rangle_{X, \textsf{F}, K,M}$ on $c_{00}$ by $$\Bigl\langle \sum_{i=1}^\infty a_i e_i\Bigr\rangle_{X, \textsf{F}, K, M}= \max\Bigl\{r_K(\sum_{i=1}^\infty a_i x_i): (x_i)_{i=1}^\infty \in B_X^\nn\cap\prod_{i=1}^\infty [F_j: m_{i-1}<j\leqslant m_i]\Bigr\},$$ where $m_0=0$ and $M=\{m_1, m_2, \ldots\}$, $m_1<m_2<\ldots$.    It is evident that for any $(a_i)_{i=1}^\infty\in c_{00}$ and any sequence $(\ee_i)_{i=1}^\infty$ of unimodular scalars, $$\Bigl\langle \sum_{i=1}^\infty a_i e_i\Bigr\rangle_{X, \textsf{F},K, M}=\Bigl\langle \sum_{i=1}^\infty a_i \ee_i e_i\Bigr\rangle_{X, \textsf{F}, K, M}$$ for any infinite subset $M$ of $\nn$.

We recall that a \emph{finite dimensional decomposition} (or \emph{FDD}) for a Banach space $X$ is a sequence $\textsf{F}=(F_n)_{n=1}^\infty$ of finite dimensional, non-zero subspaces of $X$ such that for any $x\in X$, there exists a unique sequence $(x_n)_{n=1}^\infty\in \prod_{n=1}^\infty F_n$ such that $x=\sum_{n=1}^\infty x_n$.    From this it follows that for each $n\in\nn$, the projection $P^\textsf{F}_n:X\to F_n$ given by $P^\textsf{F}_nx=x_n$, where $x=\sum_{m=1}^\infty x_m$ and $(x_m)_{n=1}^\infty\in \prod_{n=1}^\infty F_n$, is well-defined and bounded.  Furthermore, for a (finite or infinite) interval $I\subset \nn$, we let $I^\textsf{F}=\sum_{n\in I} P^\textsf{F}_n$.  It follows from the principle of uniform boundedness that $$\sup \{\|I^\textsf{F}\|: I\subset \nn\text{\ is an interval}\}<\infty.$$  We refer to this quantity as the \emph{projection constant of} $\textsf{F}$ in $X$.  If the projection constant of $\textsf{F}$ in $X$ is $1$, we say $\textsf{F}$ is \emph{bimonotone}.   It is well-known  that if $\textsf{F}$ is an FDD for $X$, then there exists an equivalent norm $|\cdot|$ on $X$ such that $\textsf{F}$ is a bimonotone FDD for $(X, |\cdot|)$.   We also remark that any FDD is also an FMD. 

If  $\textsf{F}$ is a bimonotone FDD for $X$, then $F_n^*=(P^\textsf{F}_n)^*(X^*)\subset X^*$ isometrically and canonically. Then $\textsf{F}^*:=(F^*_n)_{n=1}^\infty$ is a bimonotone  FDD for its closed span in $X^*$.   We let $X^{(*)}$ denote this closed span.     We say $\textsf{F}$ is \emph{shrinking} provided that $X^{(*)}=X^*$, which occurs if and only if any bounded block sequence with respect to $\textsf{F}$ is weakly null.    Let us note that $X^{(*)(*)}=X$.

Let  $\textsf{F}$ be an FDD for $X$.  For $x\in X$,  we let $\text{supp}_\textsf{F}(x)=\{n\in\nn: P^\textsf{F}_nx\neq 0\}$. We let $c_{00}(\textsf{F})$ denote the set of those $x\in X$ such that $\text{supp}_\textsf{F}(x)$ is finite.   We write $n<x$ (resp. $n\leqslant x$) to mean that $n<\min \text{supp}_\textsf{F}(x)$ (resp. $n\leqslant \min \text{supp}_\textsf{F}(x)$).  We write $x<y$ to mean that $\max\text{supp}_\textsf{F}(x)< \min \text{supp}_\textsf{F}(y)$.

Of course, any Schauder basis $(x_i)_{i=1}^\infty$ gives rise to the FDD $(\text{span}(x_i))_{i=1}^\infty$, and each of the definitions above for an FDD can be adapted to a Schauder basis.   In particular, if $(x_i)_{i=1}^\infty$ is a Schauder basis,  we let $E^{(*)}=[x^*_i: i\in\nn]\subset E^*$ denote the closed span of the coordinate functionals. Throughout, we let $\textsf{E}$ denote the FDD arising from the canonical $c_{00}$ basis.

We say a Banach space $E$ is a \emph{sequence space} provided that the canonical $c_{00}$ basis is a normalized  basis for $E$ having the property that for any scalar sequence $(a_i)_{i=1}^n$ and any unimodular scalars $(\ee_i)_{i=1}^n$, $$\|\sum_{i=1}^n a_ie_i\|=\|\sum_{i=1}^n a_i \ee_i e_i\|.$$   We say the sequence space $E$ has property  \begin{enumerate}[(i)]\item $R$ provided that for any strictly increasing sequences $(k_i)_{i=1}^\infty$, $(l_i)_{i=1}^\infty$ of natural numbers such that $k_i\leqslant l_i$ for each $i\in\nn$, any $n\in\nn$, and any scalars $(a_i)_{i=1}^n$, $$\|\sum_{i=1}^n a_i e_{k_i}\|\leqslant \|\sum_{i=1}^n a_i e_{l_i}\|,$$ \item $S$ provided that there exists a constant $C$ such that for any strictly increasing sequences $(k_i)_{i=1}^\infty$, $(l_i)_{i=1}^\infty$ of natural numbers such that $l_i<k_{i+1}$ for all $i\in \nn$,  any $n\in\nn$, and any scalars $(a_i)_{i=1}^n$, $$\|\sum_{i=1}^n a_i e_{l_i}\|\leqslant C\|\sum_{i=1}^n a_i e_{k_i}\|.,$$ \item $T$ provided that there exists a constant $C$ such that for any strictly increasing sequence $(k_i)_{i=1}^\infty$ of natural numbers, any $n\in \nn$,  and any sequence $(x_i)_{i=1}^n \subset X$ such that $x_i\in [e_j: k_{i-1}<j\leqslant k_i]$ (where $k_0=0$), $$\|\sum_{i=1}^n x_i\| \leqslant C \|\sum_{i=1}^n \|x_i\|e_{k_i}\|.$$   \end{enumerate}

Given a Banach space $X$ with FDD $\textsf{F}$ and a sequence space $E$, we define three quantities on $c_{00}(\textsf{F})$.  We let $$\|x\|_{X^E_\vee(\textsf{F})}=\sup\Bigl\{\|\sum_{i=1}^\infty \|I^\textsf{F}_i x\|_X e_{\max I_i}\Bigr\|_E: I_1<I_2<\ldots, I_i\text{\ an interval}\Bigr\},$$ $$[x]_{X^E_\wedge(\textsf{F})}= \inf\Bigl\{\|\sum_{i=1}^\infty \|I_i^\textsf{F}x\|_X e_{\max I_i}\|_E: I_1<I_2<\ldots, I_i\text{ an interval}, \nn=\cup_{i=1}^\infty I_i\Bigr\},$$ and $$\|x\|_{X^E_\wedge(\textsf{F})}=\inf\Bigl\{\sum_{i=1}^n [x_i]_{X^E_\wedge(\textsf{F})}: n\in \nn, x_i\in c_{00}(\textsf{F}), x=\sum_{i=1}^n x_i\Bigr\}.$$

\begin{proposition} Let $E$ be a sequence space and let $X$ be a Banach space with bimonotone FDD $\textsf{\emph{F}}$.   \begin{enumerate}[(i)]\item  $\textsf{\emph{F}}$ is a bimonotone FDD for both $X^E_\vee(\textsf{\emph{F}})$ and $X^E_\wedge(\textsf{\emph{F}})$, and $[I^\textsf{F}\cdot]_{X^E_\wedge(\textsf{\emph{F}})} \leqslant [\cdot]_{X^E_\wedge(\textsf{\emph{F}})}$ on $c_{00}(\textsf{\emph{F}})$. \item $(X^E_\vee(\textsf{\emph{F}}))^{(*)}=(X^{(*)})^{E^{(*)}}_\wedge(\textsf{\emph{F}}^*)$ and $(X^E_\vee(\textsf{\emph{F}}))^{(*)}=(X^{(*)})^{E^{(*)}}_\wedge(\textsf{\emph{F}}^*)$\end{enumerate}

\label{opposition}

\end{proposition}

\begin{proof} Throughout the proof, for ease of notation, we write $X_\vee$ and $X_\wedge$ in place of $X_\vee^E(\textsf{F})$ and $X_\wedge^E(\textsf{F})$, respectively. 

$(i)$ Let $I$ be an interval in $\nn$.   Then for any $x\in c_{00}(\textsf{F})$, \begin{align*} \|I^\textsf{F}x\|_{X_\vee} & =\sup \Bigl\{\|\sum_{i=1}^\infty \|I^\textsf{F}_i I^\textsf{F}x\|_X e_{\max I_i}\|_E\Bigr\}   = \sup \Bigl\{\|\sum_{i=1}^\infty \|I^\textsf{F}I^\textsf{F}_ix\|_X e_{\max I_i}\|_E\Bigr\} \\ &\leqslant \sup \Bigl\{\|\sum_{i=1}^\infty \|I^\textsf{F}_i x\|_X e_{\max I_i}\|_E\Bigr\}=\|x\|_{X_\vee}.\end{align*} Here, each supremum is taken over the set of all sequences of intervals $I_1<I_2<\ldots$ with $\cup_{i=1}^\infty I_i=\nn$.   Replacing the suprema above with infima gives that $[I^\textsf{F}x]_{X_\wedge}\leqslant [x]_{X_\wedge}$, and \begin{align*} \|I^\textsf{F}x\|_{X_\wedge} & = \inf\Bigl\{\sum_{i=1}^n [x_i]_{X_\wedge}: I^\textsf{F}x=\sum_{i=1}^n x_i\Bigr\}= \inf\Bigl\{\sum_{i=1}^n [I^\textsf{F}x_i]_{X_\wedge}: x=\sum_{i=1}^n x_i\Bigr\} \\ & \leqslant \Bigl\{\sum_{i=1}^n [x_i]_{X_\wedge}: x=\sum_{i=1}^n x_i\Bigr\}=\|x\|_{X_\wedge}. \end{align*}

$(ii)$ In the proof, we let $X^*_\wedge= (X^{(*)})^{E^{(*)}}_\wedge(\textsf{F}^*)$ and $(X^E_\vee(\textsf{F}))^{(*)}=(X_\vee)^{(*)}$. Fix $x\in c_{00}(\textsf{F})$, $x^*\in c_{00}(\textsf{F}^*)$, and a sequence of intervals $I_1<I_2<\ldots$ with $\cup_i I_i=\nn$.  Then \begin{align*} |x^*(x)| & \leqslant \sum_{i=1}^\infty |I^{\textsf{F}^*}_ix^*(I^\textsf{F}_i x)| \leqslant \sum_{i=1}^\infty \|I^{\textsf{F}^*}_ix^*\|_{X^*}\|I^\textsf{F}_i x\|_X \\ &  = \Bigl(\sum_{i=1}^\infty \|I^{\textsf{F}^*}_i x^*\|_{X^*} e_{\max I_i}^*\Bigr)\Bigl(\sum_{i=1}^\infty \|I^\textsf{F}_i x\|_Xe_{\max I_i}\Bigr) \\ & \leqslant \Bigl\|\sum_{i=1}^\infty \|I^{\textsf{F}^*}_i x^*\|_{X^*} e_{\max I_i}^*\Bigr\|_{E^{(*)}}\Bigl\|\sum_{i=1}^\infty \|I^\textsf{F}_i x\|_Xe_{\max I_i}\Bigr\|_E \\ & \leqslant \Bigl\|\sum_{i=1}^\infty \|I^{\textsf{F}^*}_i x^*\|_{X^*} e_{\max I_i}^*\Bigr\|_{E^{(*)}} \|x\|_{X_\vee}. \end{align*}   Taking the infium over such sequences $(I_i)_{i=1}^\infty$ yields that $|x^*(x)|\leqslant [x^*]_{X^*_\wedge}\|x\|_{X_\vee}$ for any $x^*\in c_{00}(\textsf{F}^*)$ and $x\in c_{00}(\textsf{F})$.    Now for any $x^*\in c_{00}(\textsf{F}^*)$ and $x\in c_{00}(\textsf{F})$, $$|x^*(x)| \leqslant \inf\Bigl\{\sum_{i=1}^n |x^*_i(x)|: x^*=\sum_{i=1}^n x_i^*\Bigr\} \leqslant \|x\|_{X_\vee} \inf\Bigl\{\sum_{i=1}^n [x_i^*]_{X^*_\wedge}: x^*=\sum_{i=1}^n x^*_i\Bigr\} = \|x^*\|_{X^*_\wedge}\|x\|_{X_\vee}.$$    This yields that the formal identity from $X^*_\wedge$ to $(X_\vee)^{(*)}$ is well-defined with norm $1$.   Restricting the adjoint of the formal identity to $c_{00}(\textsf{F})$ yields that the formal identity from $X_\vee=(X_\vee)^{(*)(*)}$ to $(X^*_\wedge)^{(*)}$ has norm $1$.

Now fix $x\in c_{00}$ with $\|x\|_{X_\vee}>1$.  Fix $I_1<I_2<\ldots$ such that $\|\sum_{i=1}^\infty \|I^\textsf{F}_i x\|_X e_{\max I_i}\|_E>1$.  We may fix $(a_i)_{i=1}^\infty \in c_{00}$ such that $\|\sum_{i=1}^\infty a_i e_{\max I_i}^*\|_{E^{(*)}}=1$ and $\sum_{i=1}^\infty a_i \|I^\textsf{F}_i x\|_X >1$.     We may also fix $(x^*_i)_{i=1}^\infty \in S_{X^*}^\nn$ such that $x^*_i=I^{\textsf{F}^*}_i x^*_i$ and $x^*_i(x_i)=\|x_i\|_X$ for all $i\in\nn$.   Now let $x^*=\sum_{i=1}^\infty a_i x^*_i\in c_{00}(\textsf{F}^*)$.  Note that $$\|x^*\|_{X^*_\wedge} \leqslant [x^*]_{X^*_\wedge} \leqslant \|\sum_{i=1}^\infty \|I^{\textsf{F}^*}_i x^*\|_{X^*} e_{\max I_i}^*\|_{E^{(*)}}=1$$ and $$x^*(x)= \sum_{i=1}^\infty a_i \|I^{\textsf{F}}_i x\|_X >1,$$ whence $\|x\|_{(X^*_\wedge)^{(*)}}>1$.   This yields that the formal identity from $(X^*_\wedge)^{(*)}$ to $X_\vee$ has norm $1$, and is therefore an isometric isomorphism by the last fact from the previous paragraph. Restricting the adjoint of the formal identity to $c_{00}(\textsf{F}^*)$ yields that the formal identity from $X^*_\wedge=(X^*_\wedge)^{(*)(*)}$ to $(X_\vee)^{(*)}$ has norm $1$, and is therefore also an isometric isomorphism. 

\end{proof}

\begin{rem}\upshape It follows from standard arguments that the closed unit ball of $X^E_\wedge(\textsf{F})$ is the closed, convex hull of those $x\in c_{00}(\textsf{F})$ such that $[x]_{X^E_\wedge(\textsf{F})}\leqslant 1$.   Furthermore, it follows from the fact that for any interval $I\subset \nn$, $[I^\textsf{F}\cdot]_{X^E_\wedge(\textsf{F})}\leqslant [\cdot]_{X^E_\wedge(\textsf{F})}$ that any $x\in [F_j: j\in I]$ lies in the closed convex hull of vectors $y\in [F_j: j\in I]$ such that $[y]_{X^E_\wedge(\textsf{F})} \leqslant \|x\|_{X^E_\wedge(\textsf{F})}$. 

\label{ily}

\end{rem}

\begin{lemma} Let $\textsf{\emph{F}}$ be a bimonotone FDD for the Banach space $X$ and let $E$ be a sequence space.  With $K=B_{(X^E_\wedge(\textsf{\emph{F}}))^*}$, for any infinite subset $M$ of $\nn$, $$\Bigl \langle \cdot \Bigr\rangle_{X^E_\wedge(\textsf{\emph{F}}), \textsf{\emph{F}}, K, M} \leqslant 2\Bigl\langle \cdot \Bigr\rangle_{E, \textsf{\emph{E}}, B_{E^*}, M}.$$ Furthermore, if $E$ has property $R$, the inequality holds without the factor of $2$.

\label{trip}

\end{lemma}

\begin{proof}  Using Remark \ref{ily}, it is sufficient to show that for any $(a_i)_{i=1}^\infty \in c_{00}$, any infinite subset $M$ of $\nn$, and any $(y_i)_{i=1}^\infty \in\prod_{i=1}^\infty [F_j: m_{i-1}<j\leqslant m_i]$ with $[y_i]_{X_\wedge}\leqslant 1$ for all $i\in\nn$, $$\|\sum_{i=1}^\infty a_i y_i\|_{X_\wedge}\leqslant 2\langle \sum_{i=1}^\infty a_i e_i\rangle_{E, \textsf{E}, B_{E^*}, M},$$   and that if $E$ has property $R$, the same estimate holds without the factor of $2$.

First suppose that $0\leqslant m<n$ and $0\neq y\in [F_j: m<j\leqslant n]$ is such that  $[y]_{X_\wedge}\leqslant 1$.  Then there exists a sequence $(I_i)_{i=1}^\infty$ of intervals with $I_1<I_2<\ldots$ and $\cup_{i=1}^\infty I_i=\nn$ such that $$\|\sum_{i=1}^\infty \|I_iy\|_X e_{\max I_i}\|_E\leqslant 1.$$   Let $k=\min\{i: I_iy\neq 0\}$,  $l=\max \{i: I_iy\neq 0\}$, and let $J=(m,n]$.   Note that for each $k\leqslant i<l$, $\max (J\cap I_i)=\max I_i$ and  $\max (J\cap I_l)\leqslant \max I_l$. Also, for each $i\in\nn$, $(J\cap I_i)y=I_i y$.      Furthermore, by $1$-unconditionality, \begin{align*} \|\sum_{i=1}^\infty \|I_i y\|_X e_{\max I_i}\|_E & = \|\sum_{i=k}^l \|I_i y\|_X e_{\max I_i} \|_E \leqslant \|\sum_{i=k}^{l-1} \|(J\cap I_i)y\|_X e_{\max (J\cap I_i)}\|+\|(J\cap I_l)y\|_X  \leqslant 2.\end{align*} Now if $E$ has property $R$, then \begin{align*} 1 & \geqslant \|\sum_{i=1}^\infty \|I_i y\|_X e_{\max I_i}\|_E  = \|\sum_{i=k}^l \|I_i y\|_X e_{\max I_i} \|_E \\ & \geqslant \|\sum_{i=k}^{l-1} \| (J\cap I_i)y\|_X e_{\max (J\cap I_i)} + \|(J\cap I_l)y\|_Xe_{\max(J\cap I_l)}\|.\end{align*} To summarize, if $y\in [F_j: m<j\leqslant n]$ has $[y]_{X_\wedge}\leqslant 1$, then there exist $p\in \nn$,  intervals $J_1<\ldots <J_p$ such that $\cup_{i=1}^p J_i=(m,n]$, and $v\in 2B_E\cap [e_j:m<j\leqslant n]$ such that $v=\sum_{i=1}^p \|J_i y\|_X e_{\max J_i}$, and if $E$ has property $R$, the factor of $2$ can be omitted.

Now suppose that $0=m_0<m_1<\ldots$, $y_i\in [F_j: m_{i-1}<j\leqslant m_i]$, and $[y_i]_{X_\wedge}\leqslant 1$ for all $i\in\nn$.  Applying the previous paragraph to each $y_i$ and concatenating the resulting sequences of intervals yields the existence of some sequence $I_1<I_2<\ldots$ with $\cup_{i=1}^\infty I_i=\nn$, $(v_i)_{i=1}^\infty \subset 2B_E$, and $0=k_0<k_1<\ldots$ such that for each $n\in\nn$, $(m_{n-1},m_n]=\cup_{i=k_{n-1}+1}^{k_n} I_i$ and $$v_n=\sum_{i=k_{n-1}+1}^{k_n} \|I_i y_n\|_X e_{\max I_i}.$$   Now for any $(a_i)_{i=1}^\infty \in c_{00}$, $$\|\sum_{i=1}^\infty a_iy_i\|_{X^E_\wedge(\textsf{F})} \leqslant \|\sum_{i=1}^\infty \|I_i^\textsf{F}\sum_{j=1}^\infty a_j y_j\|_X e_{\max I_i}\|_E =  \|\sum_{i=1}^\infty a_iv_i\|_E \leqslant 2\Bigl\langle \sum_{i=1}^\infty a_i e_i\Bigr\rangle_{E, \textsf{E}, B_{E^*}, M}.$$   If $E$ has property $R$, we can omit the factor of $2$.

\end{proof}

\section{Combinatorics}

Througout, we let $2^\nn$ denote the power set of $\nn$ and topologize this set with the Cantor topology. Given a subset $M$ of $\nn$, we  let $[M]$ (resp. $[M]^{<\nn}$) denote set of infinite (resp. finite) subsets of $M$.  For convenience, we often write subsets of $\nn$ as sequences, where a set $E$ is identified with the (possibly empty) sequence obtained by listing the members of $E$ in strictly increasing order.  Henceforth, if we write $(m_i)_{i=1}^r\in [\nn]^{<\nn}$ (resp. $(m_i)_{i=1}^\infty\in [\nn]$), it will be assumed that $m_1<\ldots <m_r$ (resp. $m_1<m_2<\ldots$).  Given $M=(m_n)_{n=1}^\infty\in [\nn]$ and $\mathcal{F}\subset [\nn]^{<\nn}$, we define $$\mathcal{F}(M)=\{(m_n)_{n\in E}: E\in \mathcal{F}\}$$ and $$\mathcal{F}(M^{-1})=\{E: (m_n)_{n\in E}\in \mathcal{F}\}.$$  

Given $(m_i)_{i=1}^r, (n_i)_{i=1}^r\in [\nn]^{<\nn}$, we say $(n_i)_{i=1}^r$ is a \emph{spread} of $(m_i)_{i=1}^r$ if $m_i\leqslant n_i$ for each $1\leqslant i\leqslant r$.   We agree that $\varnothing$ is a spread of $\varnothing$.    We write $E\preceq F$ if either $E=\varnothing$ or $E=(m_i)_{i=1}^r$ and $F=(m_i)_{i=1}^s$ for some $r\leqslant s$.   In this case, we say $E$ is an \emph{initial segment} of $F$.  For $E,F\subset \nn$, we write $E<F$ to mean that either $E=\varnothing$, $F=\varnothing$, or $\max E<\min F$. Given $n\in \nn$ and $E\subset \nn$, we write $n\leqslant E$ (resp. $n<E$) to mean that $n\leqslant \min E$ (resp. $n<\min E$).

We say $\mathcal{G}\subset [\nn]^{<\nn}$ is \begin{enumerate}[(i)]\item \emph{compact} if it is compact in the Cantor topology, \item \emph{hereditary} if $E\subset F\in \mathcal{G}$ implies $E\in \mathcal{G}$,  \item \emph{spreading} if whenever $E\in \mathcal{G}$ and $F$ is a spread of $E$, $F\in \mathcal{G}$,  \item \emph{regular} if it is compact, hereditary, and spreading. \end{enumerate}

Given a topological space $K$ and a subset $L$ of $K$, $L'$ denotes the \emph{Cantor Bendixson derivative} of $L$, which consists of those members of $L$ which are not relatively isolated in $L$.    We define by transfinite induction the higher order transfinite derivatives of $L$ by $$L^0=L,$$ $$L^{\xi+1}=(L^\xi)',$$ and if $\xi$ is a limit ordinal, $$L^\xi=\bigcap_{\zeta<\xi}L^\zeta.$$ We recall that $K$ is said to be \emph{scattered} if there exists an ordinal $\xi$ such that $K^\xi=\varnothing$. In this case, we define the \emph{Cantor Bendixson index} of $K$ by $CB(K)=\min\{\xi: K^\xi=\varnothing\}$.    If $K^\xi\neq \varnothing$ for all ordinals $\xi$, we write $CB(K)=\infty$.   We agree to the convention that $\xi<\infty$ for all ordinals $\xi$, and therefore $CB(K)<\infty$ simply means that $CB(K)$ is an ordinal, and $K$ is scattered. 

Of course, if $\xi$ is a limit ordinal, $K$ is a compact  topological space, and $K^\zeta\neq \varnothing$ for all $\zeta<\xi$, then $(K^\zeta)_{\zeta<\xi}$ is a collection of compact subsets of $K$ with the finite intersection property, so $K^\xi=\cap_{\zeta<\xi}K^\zeta\neq \varnothing$.   From this it follows that for a compact topological space, $CB(K)$ cannot be a limit ordinal.

We recall the following, which is well known.  The proof is standard, so we omit it.

\begin{fact} Let $\mathcal{G}\subset [\nn]^{<\nn}$ be hereditary.   The following are equivalent.  \begin{enumerate}[(i)]\item There does not exist $M\in [\nn]$ such that $[M]^{<\nn}\subset \mathcal{G}$.\item $\mathcal{G}$ is compact.  \item $CB(\mathcal{G})<\infty$. \item $CB(\mathcal{G})<\omega_1$. \end{enumerate}

\end{fact}

For each $n\in\nn$, we let $\mathcal{A}_n=\{E\in [\nn]^{<\nn}: |E|\leqslant n\}$.    It is clear that $\mathcal{A}_n$ is regular.    Also of importance are the Schreier families, $(\mathcal{S}_\xi)_{\xi<\omega_1}$.  We recall these families.  We let $$\mathcal{S}_0=\mathcal{A}_1,$$ $$\mathcal{S}_{\xi+1}=\{\varnothing\} \cup \Bigl\{\bigcup_{i=1}^n E_i: \varnothing\neq E_i \in \mathcal{S}_\xi, n\leqslant E_1<\ldots <E_n\Bigr\},$$ and if $\xi<\omega_1$ is a limit ordinal, there exists a sequence $\xi_n\uparrow \xi$ such that $$\mathcal{S}_\xi=\{E\in [\nn]^{<\nn}: \exists n\leqslant E\in \mathcal{S}_{\xi_n+1}\}.$$  We note that the sequence $(\xi_n)_{n=1}^\infty$ has the property that for any $n\in\nn$, $\mathcal{S}_{\xi_n+1}\subset \mathcal{S}_{\xi_{n+1}}$.    The existence of such families with the last indicated property is discussed, for example, in \cite{Concerning}.  

Given two non-empty regular families $\mathcal{F}, \mathcal{G}$, we let $$\mathcal{F}[\mathcal{G}]=\{\varnothing\}\cup \Bigl\{\bigcup_{i=1}^n E_i: \varnothing \neq E_i\in \mathcal{G}, E_1<\ldots <E_n, (\min E_i)_{i=1}^n\in \mathcal{F}\Bigr\}.$$  We let $\mathcal{F}[\mathcal{G}]=\varnothing$ if either $\mathcal{F}=\varnothing$ or $\mathcal{G}=\varnothing$. 

Given a regular family $\mathcal{G}$, we let $MAX(\mathcal{G})$ denote the set of maximal members of $\mathcal{G}$ with respect to inclusion (noting that this is also the set of maximal members of $\mathcal{G}$ with respect to the initial segment ordering).  We note that for each $\xi<\omega_1$ and any $\varnothing\neq E\in \mathcal{S}_\xi$, either $E\in MAX(\mathcal{S}_\xi)$ or $E\cup (1+\max E)\in \mathcal{S}_\xi$.   From this it follows that for any $M=(m_i)_{i=1}^\infty\in [\nn]$, there exist unique $0=k_0<k_1<\ldots$ such that for each $i\in\nn$, $(m_j)_{j=k_{i-1}+1}^{k_i}\in MAX(\mathcal{S}_\xi)$.  We define $M_{\mathcal{S}_\xi}=(m_j)_{j=1}^{k_1}$, and $M_{\mathcal{S}_\xi, i}=(m_j)_{j=k_{i-1}+1}^{k_i}$.

The following facts are collected in \cite{Concerning}.

\begin{proposition} \begin{enumerate}[(i)]\item For any non-empty regular families $\mathcal{F}, \mathcal{G}$, $\mathcal{F}[\mathcal{G}]$ is regular. Furthermore, if $CB(\mathcal{F})=\beta+1$ and $CB(\mathcal{G})=\alpha+1$, then $CB(\mathcal{F}[\mathcal{G}])=\alpha\beta+1$. \item For any $n\in\nn$, $CB(\mathcal{A}_n)=n+1$. \item For any $\xi<\omega_1$, $CB(\mathcal{S}_\xi)=\omega^\xi+1$. \item If $\mathcal{F}$ is regular and $M\in [\nn]$, then $\mathcal{F}(M^{-1})$ is regular and $CB(\mathcal{F})=CB(\mathcal{F}(M^{-1}))$. \item For regular families $\mathcal{F}, \mathcal{G}$, there exists $M\in [\nn]$ such that $\mathcal{F}(M)\subset \mathcal{G}$ if and only if there exists $M\in [\nn]$ such that $\mathcal{F}\subset \mathcal{G}(M^{-1})$ if and only if $CB(\mathcal{F})\leqslant CB(\mathcal{G})$. \end{enumerate} 

\label{deep facts}

\end{proposition}

For a probability measure $\mathbb{P}$ on $\nn$, we write $\mathbb{P}(n)$ to mean $\mathbb{P}(\{n\})$.   Furthermore, we let $\text{supp}(\mathbb{P})=\{n\in\nn: \mathbb{P}(n)>0\}$.  We will recall the repeated averages hierarchy, introduced in \cite{AMT}.  For each countable ordinal $\xi$,  we will define a collection $\mathfrak{S}_\xi=\{\mathbb{S}^\xi_{M,n}: M\in [\nn], n\in\nn\}$ of probability measures on $\nn$.    If $M=(m_i)_{i=1}^\infty$, we let $\mathbb{S}^0_{M,n}=\delta_{m_n}$, the Dirac measure at $m_n$.    If $\mathfrak{S}_\xi$ has been defined and $M\in[\nn]$, we let $M_1=M$, $p_0=s_0=0$, $p_1=\min M_1$.    Now assume that $M_1, \ldots, M_{n-1}$, $s_0, \ldots, s_{n-1}$, $ \ldots, p_{n-1}$, and $\mathbb{S}^{\xi+1}_{M, 1}, \ldots, \mathbb{S}^{\xi+1}_{M,n-1}$ have been defined such that $\mathbb{S}^\xi_{M,i}=p_i^{-1}\sum_{j=s_{i-1}+1}^{s_i}  \mathbb{S}^\xi_{M,j}$ and $s_i=s_{i-1}+p_i$.  Let $M_n=M\setminus \cup_{j=1}^{s_{n-1}} \text{supp}(\mathbb{S}^\xi_{M,j})$, $p_n=\min M_n$, $s_n= s_{n-1}+p_n$, and $\mathbb{S}^{\xi+1}_{M, n}=\sum_{j=s_{n-1}+1}^{s_n} \mathbb{S}^\xi_{M,j}$.  Now assume that $\xi$ is a countable limit ordinal and $\mathfrak{S}_\zeta$ has been defined for each $\zeta<\xi$.   Let $(\xi_n)_{n=1}^\infty$ be the sequence such that $$\mathcal{S}_\xi=\{\varnothing\}\cup \{E\in [\nn]^{<\nn}: \varnothing\neq E\in \mathcal{S}_{\xi_{\min E}+1}\}.$$     Then let $M_1=M$, $p_1=\min M_1$, and $\mathbb{S}^\xi_{M,1}=\mathbb{S}^{\xi_{p_1}+1}_{M_1, 1}$.   Now assuming that $M_1, \ldots, M_{n-1}$, $p_1, \ldots, p_{n-1}$, and $\mathbb{S}^\xi_{M,1}, \ldots, \mathbb{S}^\xi_{M,n-1}$ have been defined, let $M_n=M\setminus \cup_{i=1}^{n-1}\text{supp}(\mathbb{S}^\xi_{M,i})$, $p_n=\min M_n$, and $\mathbb{S}^\xi_{M,n}= \mathbb{S}^{\xi_{p_n}+1}_{M_n, 1}$.

We isolate the following properties of the collections $\mathfrak{S}_\xi$, shown in \cite{AMT}. 

\begin{proposition}\begin{enumerate}[(i)]\item For each ordinal $\xi$, each $M\in[\nn]$, and each $n\in\nn$, $\text{\emph{supp}}(\mathbb{S}^\xi_{M,n})=M_{\mathcal{S}_\xi}$. \item If $M,N\in[\nn]$ and $r_1<\ldots <r_k$ are such that $\cup_{i=1}^k \text{\emph{supp}}(\mathbb{S}^\xi_{M,r_i})$ is an initial segment of $N$, then $\mathbb{S}^\xi_{N,i}=\mathbb{S}^\xi_{M,r_i}$ for each $1\leqslant i\leqslant k$. \end{enumerate}

\end{proposition} 

The second property above is called the \emph{permanence property}.   

Let us recall the following result of Gasparis.

\begin{theorem}\cite{Ga} If $\mathcal{F}, \mathcal{G}\subset [\nn]^{<\nn}$ are hereditary, then for any $M\in [\nn]$, there exists $N\in [M]$ such that either $$\mathcal{F}\cap [N]^{<\nn}\subset \mathcal{G}\text{\ \ \ \ \ or\ \ \ \ \ }\mathcal{G}\cap [N]^{<\nn}\subset \mathcal{F}.$$  

In particular, if $\mathcal{G}$ is regular and $CB(\mathcal{F})<CB(\mathcal{G})$, then for any $M\in [\nn]$, there exists $N\in [M]$ such that $\mathcal{F}\cap [N]^{<\nn}\subset \mathcal{G}$. 
\label{gasp}
\end{theorem}

The first statement was proved directly in \cite{Ga}, while the second follows from the fact that for any regular $\mathcal{G}$ and $N\in[\nn]$, $CB(\mathcal{G}\cap [N]^{<\nn})=CB(\mathcal{G})$.

We also will need the following, shown in \cite{CN}.

\begin{proposition} \begin{enumerate}[(i)]\item For any countable ordinal $\xi$, if $\mathcal{H}$ is regular with $CB(\mathcal{H})\leqslant \omega^\xi+1$,  and $q\in \nn$, then for any $\ee>0$ and $M\in[\nn]$, there exists $N\in[M]$ such that $$\sup \{\mathbb{S}^\xi_{P,1}(E): E\in \mathcal{H}, \min E\leqslant q, P\in[N]\}\leqslant \ee.$$ \item  If $\xi<\omega_1$ and if $\mathcal{H}$ is a regular family with $CB(\mathcal{H})\leqslant \omega^\xi$, then for any $\ee>0$ and $M\in[\nn]$, there exists $N\in[M]$ such that $$\sup \{\mathbb{S}^\xi_{P,1}(E): E\in \mathcal{H}, P\in [N]\}\leqslant \ee.$$\end{enumerate}

\label{CN}
\end{proposition}

Given a regular family $\mathcal{G}$ and $M\in[\nn]$, let $\mathcal{G}\bowtie M=\{(i,F): i\in F\in [M]^{<\nn}\cap MAX(\mathcal{G})\}$.  Given  a function $f:\mathcal{G}\bowtie M\to \rr$ and $N\in[M]$, we let $$\|f\|_N=\sup \{|f(i,F)|: (i,F)\in \mathcal{G}\bowtie N\}.$$   
The next result combines an argument of Schlumprecht (\cite[Corollary $4.10$]{Schlumprecht} with \cite[Lemma $3.10$]{CN}.

\begin{lemma} Fix a countable ordinal $\xi$ and $\ee\in \rr$.    Let $f:\mathcal{S}_\xi\bowtie Q\to \rr$ be  a bounded function.    If there exists $L\in[Q]$ such that $$[L]\subset \Bigl\{M\in[\nn]: \sum_{j\in M_{\mathcal{S}_\xi}} f(j, M_{\mathcal{S}_\xi}) \geqslant \ee\Bigr\},$$ then for any $M\in [L]$ and $\delta<\ee$,  there exists $P\in [M]$ such that for any $E\in \mathcal{S}_\xi$, there exists $F\in MAX(\mathcal{S}_\xi)\cap [M]^{<\nn}$ such that $P(E)\subset F$ and for each $j\in P(E)$, $f(j, F)\geqslant \delta$.

\label{Schlumprecht}
\end{lemma}

We next recall a special case of the infinite Ramsey theorem, the proof of which was achieved in steps by Nash-Williams \cite{NW}, Galvin and Prikry \cite{GP}, Silver \cite{Silver}, and Ellentuck \cite{Ellentuck}.

\begin{theorem} If $\mathcal{V}\subset [\nn]$ is closed, then for any $M\in [\nn]$, there exists $N\in [M]$ such that either $$[N]\subset \mathcal{V}\text{\ \ \ \ \ or\ \ \ \ \ }[N]\cap \mathcal{V}=\varnothing.$$

\end{theorem}

\section{Schreier, mixed Schreier, and Baernstein spaces}

Given $F\subset \nn$, we let $F$ denote the projection from $c_{00}$ to itself given by $Fx=(1_F(i)e^*_i(x))_{i=1}^\infty$.   Given a regular family $\mathcal{G}$ containing all singletons,  we let $X_\mathcal{G}$ be the completion of $c_{00}$ with respect to the norm $$\|x\|_\mathcal{G}= \max\{\|Fx\|_{\ell_1}: F\in \mathcal{G}\}.$$ These are the \emph{Schreier spaces}.      Given $1<p\leqslant \infty$, we let $X_{\mathcal{G},p}$ denote the completion of $c_{00}$ with respect to the norm $$\|x\|_{\mathcal{G},p}= \sup\Bigl\{\|\sum_{i=1}^\infty \|F_ix\|_{\ell_1}e_i\|_{\ell_p}: F_1<F_2<\ldots, F_i\in \mathcal{G}\Bigr\}.$$  These are the \emph{Baernstein spaces}. For convenience, if $\mathcal{G}=\mathcal{S}_\xi$, we write $\|\cdot\|_\xi$ in place of $\|\cdot\|_{\mathcal{S}_\xi}$ and we write $\|\cdot\|_{\xi,p}$ in place of $\|\cdot\|_{\mathcal{S}_\xi, p}$.

Given a sequence $\mathcal{G}_0, \mathcal{G}_1, \ldots$ of regular families such that $\mathcal{G}_0$ contains all singletons and a sequence $1=\vartheta_0>\vartheta_1>\ldots$ with $\lim_n \vartheta_n=0$, we let $X(\mathcal{G}_n, \vartheta_n)$ denote the completion of $c_{00}$ with respect to the norm $$\|x\|_{\mathcal{G}_n, \vartheta_n}=\sup \{\vartheta_n \|x\|_{\mathcal{G}_n}: n\in\nn\cup \{0\}\}.$$   We will refer to these spaces as the \emph{mixed Schreier} spaces.   Note that the Schreier, Baernstein, and mixed Schreier spaces have properties $R$ and $S$.  Note also that the Schreier and Baernstein spaces satisfy property $T$.  

\begin{lemma} Fix $\xi<\omega_1$, $1<p\leqslant \infty$, regular families $\mathcal{G}_0, \mathcal{G}_1, \ldots$, and a null sequence $(\vartheta_n)_{n=0}^\infty$ such that $1=\vartheta_0>\vartheta_1>\ldots$. Let $1/p+1/q=1$.   \begin{enumerate}[(i)]\item If $0<\ee\leqslant 1$ and $m\in\nn$ are such that $1/m^{1/q}<\ee$, then $$CB(\{F\in [\nn]^{<\nn}: (\forall x\in \text{\emph{co}}(e_i: i\in F))(\|x\|_{\xi,p} \geqslant \ee)\})\leqslant \omega^\xi m.$$ \item $CB(\{F\in [\nn]^{<\nn}: (\forall x\in \text{\emph{co}}(e_i: i\in F))(\|x\|_{\xi,p} \geqslant 1/m^{1/q})\})= \omega^\xi m+1.$    \item If $0<\xi$ and  $CB(\mathcal{G}_n)\leqslant \omega^\xi$ for all $n\in\nn$, then for any $0<\ee\leqslant 1$, $$CB(\{F\in [\nn]^{<\nn}: (\forall x\in \text{\emph{co}}(e_i: i\in F))(\|x\|_{\mathcal{G}_n, \vartheta_n}\geqslant \ee)\})<\omega^\xi.$$   \end{enumerate}

\label{door}

\end{lemma}

\begin{proof} $(i)$ Fix $1<p\leqslant \infty$ and for $0<\ee\leqslant 1$, let $$\mathcal{B}_\ee=\{E\in [\nn]^{<\nn}: (\forall x\in \text{co}(e_i: i\in F))(\|x\|_{\xi,p}\geqslant \ee)\}.$$ It is clear that $\mathcal{B}_\ee$ is hereditary, and since $X_{\xi,p}$ has property $R$, $\mathcal{B}_\ee$ is spreading.   Fix $0<1/m^{1/q}<\ee\leqslant 1$ and suppose that $CB(\mathcal{B}_\ee)\geqslant \omega^\xi m$.   If $\xi=0$, then $X_{\xi,p}=\ell_p $ (resp. $c_0$ if $p=\infty$).  Then if $E\in \mathcal{B}_\ee$ and $r=|E|>0$, then $$\ee\leqslant \|r^{-1}\sum_{i\in E}e_i\|_{\ell_p} = 1/r^{1/q},$$ and $r<m$.   This means $CB(\mathcal{B}_\ee)<m+1$, and $CB(\mathcal{B}_\ee)\leqslant m$.

Now suppose $0<\xi$.  Then if $CB(\mathcal{B}_\ee)\geqslant \omega^\xi m$,  since $\omega^\xi m$ is a limit ordinal, $CB(\mathcal{B}_\ee)\geqslant \omega^\xi m+1$. This means there exists $M=(m_i)_{i=1}^\infty \in [\nn]$ such that $\mathcal{A}_m[\mathcal{S}_\xi](M)\subset \mathcal{B}_\ee$.    Let $\mathcal{G}=\mathcal{S}_\xi(M^{-1})$ and note that $CB(\mathcal{G})=\omega^\xi+1$.   Fix $\delta>0$ such that $\ee>1/m^{1/q}+m\delta$.   Recursively select $N_1=\nn$, $n_1\in N_1$, $N_2\in [N_1]$, $n_1<n_2\in N_2$, $\ldots$, such that for each $k\in\nn$, $$\sup \{\mathbb{S}^\xi_{N,1}(E): E\in \mathcal{G}, \min E\leqslant n_{k-1}, N\in [N_k]\}\leqslant \delta.$$  Now let $N=(n_i)_{i=1}^\infty$.

 For each $i\in\nn$, let $x_i=\sum_{j=1}^\infty \mathbb{S}^\xi_{N,i}(j)e_{m_j}$ and Let $P_i=M\setminus \cup_{j=1}^{i-1}\text{supp}(\mathbb{S}^\xi_{N,j})$. Suppose that for some $F\in \mathcal{S}_\xi$ and $i\in \nn$, $Fx_i\neq 0$.  Fix $i<j$.   Let $\max\text{supp}(x_i)=m_{n_{s-1}}$ and let $G=(n_{s-1})\cup \{i: m_i\in F\cap \text{supp}(x_j)\}$.  Then $M(G)$ is a spread of a subset of $F$, so $G\in \mathcal{G}$ and $\min G\leqslant n_{s-1}$. Furthermore, since $P_s\in [N_s]$,   $$\|Fx_j\|_{\ell_1} \leqslant \mathbb{S}^\xi_{N,s}(G)=\mathbb{S}^\xi_{P_s,1}(G)\leqslant \delta.$$

  Note that $\cup_{i=1}^m \text{supp}(x_i)\in \mathcal{A}_m[\mathcal{S}_\xi](M)\subset \mathcal{B}_\ee$.    From this it follows that with $x=m^{-1}\sum_{i=1}^m x_i$,  $\|x\|_{\xi,p}\geqslant \ee$, whence there exist $F_1<\ldots <F_r$, $F_j\in \mathcal{S}_\xi$ such that $$\ee\leqslant \|\sum_{j=1}^r \|F_jx\|_{\ell_1}e_j\|_{\ell_p}.$$  By omitting extraneous sets, we may assume that $F_jx\neq \varnothing$ for each $1\leqslant j\leqslant r$.  Let $T_1, \ldots, T_m$ be such that $j\in T_i$ if and only if $F_jx_l=\varnothing$ for each $l<i$ and $F_jx_i\neq \varnothing$.   Note that for each $1\leqslant i<m$ and $j\in T_i\setminus \{\max T_i\}$, $F_j x_l=0$ for each $i<l\leqslant m$, and if $j=\max T_i$, $\|F_jx_l\|_{\ell_1}\leqslant \delta$ for each $i<l\leqslant m$.      From this it follows that \begin{align*} \|\sum_{j=1}^r \|F_jx\|_{\ell_1} e_j\|_{\ell_p} & \leqslant \|\sum_{i=1}^m \frac{1}{m}(\sum_{j\in T_i}\|F_jx_i\|_{\ell_1}+ \delta m)e_j\|_{\ell_p} \leqslant \frac{1}{m}\|\sum_{i=1}^m \|x_i\|_{\ell_1} e_j\|_{\ell_p} + \delta m = 1/m^{1/q}+\delta m<\ee.  \end{align*}

$(ii)$ If $\xi=0$, then $\mathcal{B}_{1/m^{1/q}}=\mathcal{A}_m$. Now assume $0<\xi$.  It is easy to verify that $\mathcal{A}_m[\mathcal{S}_\xi]\subset \mathcal{B}_{1/m^{1/q}}$, whence $CB(\mathcal{B}_{1/m^{1/q}})\geqslant CB(\mathcal{A}_m[\mathcal{S}_\xi])=\omega^\xi m+1$.    Seeking a contradiction, assume $CB(\mathcal{B}_{1/m^{1/q}})>\omega^\xi m+1$.  This means there exists $n_0$ such that $\mathcal{H}=\{E\in [\nn]^{<\nn}: (n_0)\cup E\in \mathcal{B}_{1/m^{1/q}}\}$ has $CB(\mathcal{H})\geqslant \omega^\xi m+1$.  From this it follows that there exists $M\in [\nn]$ such that $\mathcal{A}_m[\mathcal{S}_\xi](M)\subset \mathcal{H}$.       Let $\mathcal{G}=\mathcal{H}(M^{-1})$.   Arguing as above, we fix $\delta>0$ such that $1/(m+1)^{1/q}+ \delta (m+1)<1/m^{1/q}$.  We then recursively select $N_1$, $n_0<n_1\in N_1$, $N_2\in [N_1]$, $n_1<n_2\in N_2$, $\ldots$ such that for each $k\in\nn$, $$\sup \{\mathbb{S}^\xi_{P, 1}(E): E\in \mathcal{H}, \min E\leqslant n_{k-1}, E\in \mathcal{G}, P\in [N_k]\}\leqslant \delta.$$  Let $N=(n_i)_{i=1}^\infty$.     We argue as in $(i)$ to deduce that $$\frac{1}{m^{1/q}} \leqslant \|\frac{1}{m+1}(e_{n_0}+\sum_{i=1}^m \sum_{j=1}^\infty \mathbb{S}^\xi_{N,i}(j)e_{m_j})\|_{\xi,p} \leqslant \frac{1}{(m+1)^{1/q}} + \delta(m+1),$$ a contradiction.

$(iii)$ Let $X=X(\mathcal{G}_n, \vartheta_n)$ and for $0<\ee\leqslant 1$, let $\mathcal{B}_\ee=\{E\in [\nn]^{<\nn}: (\forall x\in \text{co}(e_i: i\in E)(\|x\|_X\geqslant \ee)\}$. Note that, since $X$ has property $R$, $\mathcal{B}$ is spreading and hereditary.  If $CB(\mathcal{B}_\ee)\geqslant \omega^\xi$, then since $\omega^\xi$ is a limit ordinal, $CB(\mathcal{B}_\ee)>\omega^\xi$.  Then there exists $M\in[\nn]$ such that $\mathcal{S}_\xi(M)\subset \mathcal{B}_\ee$.   Fix $n_1\in \nn$ such that $\vartheta_{n_1}<\ee$ and $N\in[\nn]$ such that $$\sup \{\mathbb{S}^\xi_{N,1}(E): E\in \cup_{i=1}^{n_1} \mathcal{G}_i(M^{-1})\}<\ee/2.$$   We may do this, since $$CB(\cup_{i=1}^{n_1} \mathcal{G}_i(M^{-1}))=\max_{1\leqslant i\leqslant n_1}CB(\mathcal{G}_i)<\omega^\xi.$$   Then let $x=\sum_{j=1}^\infty \mathbb{S}^\xi_{N,1}(j)e_{m_j}$ and note that, since $\text{supp}(x)\in \mathcal{B}_\ee$, $\|x\|_X\geqslant \ee$.  However, if $F\in \cup_{i=1}^{n_1} \mathcal{G}_i$ and $G=\{i: m_i\in F\cap \text{supp}(x)\}$, $$\|Fx\|_{\ell_1}\leqslant \mathbb{S}^\xi_{N,1}(G)\leqslant \ee/2.$$   Thus $$\ee\leqslant \|x\|_X \leqslant \max\{\ee/2, \sup_{n\geqslant n_1}\vartheta_n \|x\|_{\ell_1}\} \leqslant \max\{\ee/2, \vartheta_{n_1}\}<\ee,$$ a contradiction.

\end{proof}

Fix $\xi\in (0, \omega_1)\setminus \{\omega^\eta: \eta\text{\ a limit ordinal}\}$.  If $\xi=\omega^{\zeta+1}$, let us say that the mixed Schreier space $X(\mathcal{G}_n, \vartheta_n)$ is $\xi$-\emph{well-constructed} provided that there exist $0<\vartheta<1$ and a regular family $\mathcal{G}$ with $\omega^{\omega^\zeta}<CB(\mathcal{G})<\omega^{\omega^{\zeta+1}}$ such that $$\mathcal{G}_0=\mathcal{S}_0,$$ $$\mathcal{G}_n=\mathcal{G}[\mathcal{G}_{n-1}]$$ for $n\in\nn$, and $\vartheta_n=\vartheta^n$ for all $n \in \nn\cup \{0\}$.   Note that such a sequence exists.  Indeed, we may take $\mathcal{G}_0=\mathcal{S}_\beta$ for some $\omega^\zeta<\beta<\omega^{\zeta+1}$ and then $\mathcal{G}_n=\mathcal{S}_\beta[\mathcal{G}_{n-1}]$.    

 If $\xi=1$, let us say that the mixed Schreier space $X(\mathcal{G}_n, \vartheta_n)$ is $\xi$-\emph{well-constructed} provided that there exist $0<\vartheta<1$ and a regular family $\mathcal{G}$ with $1<CB(\mathcal{G})<\omega$ such that $$\mathcal{G}_0=\mathcal{S}_0,$$ $$\mathcal{G}_n=\mathcal{G}[\mathcal{G}_{n-1}]$$ for $n\in\nn$, and $\vartheta_n=\vartheta^n$ for all $n \in \nn\cup \{0\}$.  Note  that such a sequence $\mathcal{G}_0, \mathcal{G}_1, \ldots$ exists. Indeed, we may fix $l\in \nn$ and take $\mathcal{G}_n=\mathcal{A}_{l^n}$ for all $n\in \nn\cup \{0\}$. 

Now assume that $\xi\in (0,\omega_1)\setminus \{\omega^\eta: \eta<\omega_1\}$.    Let us say $X(\mathcal{G}_n, \vartheta_n)$ is $\xi$-\emph{well-constructed} provided that there exist some ordinals $\beta, \gamma<\xi$ such that $\beta+\gamma=\xi$, $CB(\mathcal{G}_0)=\omega^\beta+1$ and there exist regular families $\mathcal{F}_1, \mathcal{F}_2, \ldots$ such that $\mathcal{G}_n=\mathcal{G}_n[\mathcal{G}_0]$ and $CB(\mathcal{F}_n)\uparrow \omega^\gamma$. Note that there is no requirement that $(\vartheta_n)_{n=0}^\infty$   be a geometric sequence in this case.    Note that such $\beta, \gamma$ and such a sequence of $\mathcal{G}_0, \mathcal{G}_1, \ldots$ exists.  Indeed, by basic facts about ordinals, if $\xi\in (0,\omega_1)\setminus \{\omega^\eta: \eta<\omega_1\}$, there exist $\beta, \gamma<\xi$ with $\beta+\gamma=\xi$. If $\gamma=\zeta+1$, let $\mathcal{G}_0=\mathcal{S}_\beta$, $m_1<m_2<\ldots$ be natural numbers, and $\mathcal{F}_n=\mathcal{A}_{m_n}[\mathcal{S}_\zeta]$.   If $\gamma$ is a limit ordinal, let $\mathcal{G}_0=\mathcal{S}_\beta$, $\gamma_n\uparrow \gamma$, and $\mathcal{F}_n=\mathcal{S}_{\gamma_n}$.

\section{Szlenk index}

Given a Banach space $X$, a weak$^*$-compact subset $K$ of $X^*$, and $\ee>0$, we let $s_\ee(K)$ denote the set of those $x^*\in K$ such that for any weak$^*$-neighborhood $v$ of $x^*$, $\text{diam}(v\cap K)>\ee$.  We let $s_\ee(K)=K$ for any $\ee\leqslant 0$.   We then define the transfinite derivations by $$s_\ee^0(K)=K,$$ $$s^{\xi+1}_\ee(K)=s_\ee(s^\xi_\ee(K)),$$ and if $\xi$ is a limit ordinal, let $$s^\xi_\ee(K)=\bigcap_{\zeta<\xi} s^\zeta_\ee(K).$$   If there exists an ordinal $\xi$ such that $s^\xi_\ee(K)=\varnothing$, we let $Sz(K, \ee)$ be the minimum such ordinal, and otherwise we write $Sz(K,\ee)=\infty$.   We agree to the convention that $Sz(K, \ee)<\infty$ means there exists an ordinal $\xi$ such that $s^\xi_\ee(K)=\varnothing$.  If $Sz(K, \ee)<\infty$ for all $\ee>0$, then we let $Sz(K)=\sup_{\ee>} Sz(K, \ee)$, and otherwise we write $Sz(K)=\infty$.    If $A:X\to Y$ is an operator, we write $Sz(A, \ee)$ and $Sz(A)$ in place of $Sz(A^*B_{Y^*}, \ee)$ and $Sz(A^*B_{Y^*})$, respectively.    If $X$ is a Banach space, we write $Sz(X,\ee)$ and $Sz(X)$ in place of $Sz(I_X,\ee)$ and $Sz(I_X)$.

If $K\subset X^*$ is weak$^*$-compact and $Sz(K)<\infty$, then for any $\ee>0$, there exists a minimum ordinal $\zeta$ such that $s^{\omega^\xi \zeta}_\ee(K)=\varnothing$.  We let $Sz_\xi(K, \ee)$ be this minimum ordinal. If $Sz(K, \ee)=\infty$, then $Sz_\xi(K, \ee)=\infty$.     If $Sz(K)\leqslant \omega^{\xi+1}$, then $Sz_\xi(K, \ee)<\omega$ for each $\ee>0$.   We then define $$\textsf{p}_\xi(K)=\underset{\ee\to 0^+}{\lim\sup} \frac{\log Sz_\xi(K, \ee)}{|\log(\ee)|},$$ noting that this value need not be finite.  If $Sz(K)=\infty$ or $Sz(K)>\omega^{\xi+1}$, we let $\textsf{p}_\xi(K)=\infty$.

The following is a generalization of a result from \cite{CL}.

\begin{lemma} Fix $\xi\in (0,\omega_1)\setminus \{\omega^\eta: \eta\text{\ a limit ordinal}\}$ and let $X=X(\mathcal{G}_n, \vartheta_n)$ be a $\xi$-well-constructed mixed Schreier space. \begin{enumerate}[(i)]\item If $\vartheta_n<\ee$, $Sz(X, \ee)\geqslant CB(\mathcal{G}_n)$.  \item $Sz(X)=\omega^\xi$. \end{enumerate}
\label{ww}
\end{lemma}

\begin{proof}$(i)$ It is straightforward to see that $\vartheta_n\sum_{i\in F} e^*_i\in B_{X^*}$ for any $F\in \mathcal{G}_n$, and if $F,G\in \mathcal{G}_n$ are distinct, $\|\vartheta_n\sum_{i\in F}e^*_i-\vartheta_n\sum_{i\in G} e^*_i\|\geqslant \vartheta_n$.   Furthermore, if $(F_j)_{j=1}^\infty \subset \mathcal{G}_n$ and $F_j\to F$ in the Cantor topology, then $\vartheta_n \sum_{i\in F_j} e^*_i\underset{\text{weak}^*}{\to}\vartheta_n \sum_{i\in F}e^*_i$.  From this and an easy induction argument it follows that for every ordinal $\eta$, $\{\vartheta_n\sum_{i\in F}e_i^*: F\in \mathcal{G}_n^\eta\}\subset s^\eta_\ee(B_{X^*})$.   In particular, if $CB(\mathcal{G}_n)=\zeta_n+1$, then $0=\vartheta_n\sum_{i\in \varnothing} e^*_i\in s^{\zeta_n}_\ee(B_{X^*})$ and $Sz(X, \ee)\geqslant \zeta_n+1=CB(\mathcal{G}_n)$.

$(ii)$ Part $(i)$ yields that $Sz(B_{X^*})\geqslant  \omega^\xi$.    We focus on the reverse estimate. Let $K_n=\{ \vartheta_n \sum_{i\in F}e^*_i: F\in \mathcal{G}_n\}$ and let $K=\cup_{n=0}^\infty K_n$.   Note that there exists $r>0$ such that  $rB_{X^*}\subset \overline{\text{abs\ co}}^{\text{weak}^*}(K)$ (we may take $r=1/2$ if $\mathbb{K}=\mathbb{R}$ and $r=1/2\sqrt{2}$ if $\mathbb{K}=\mathbb{C}$).   From this it follows that $Sz(X)=Sz(B_{X^*})=Sz(rB_{X^*}) \leqslant Sz(\overline{\text{abs\ co}}^{\text{weak}^*}(K))$.  By the main theorem of \cite{C3}, $Sz(\overline{\text{abs\ co}}^{\text{weak}^*}(K))\leqslant \omega^\xi$ if $Sz(K)\leqslant \omega^\xi$, whence it is sufficient to prove that $Sz(K)\leqslant \omega^\xi$.   Note also that $K$ and $K_n$ are weak$^*$-compact. For any ordinal $\eta$ and any $\ee>0$, $$s^\eta_\ee(K)\subset \{0\}\cup \bigcup_{n=0}^\infty s_\ee^\eta(K).$$   Thus it suffices to show that for any $\ee>0$, $\sup_n Sz(K_n, \ee)<\omega^\xi$.    

We first note that for any $n\in \nn\cup \{0\}$, any $\ee>0$, and any ordinal $\eta$, $$s_\ee^\eta(K_n)\subset \{\vartheta_n \sum_{i\in F}e^*_i: F\in \mathcal{G}_n^\eta\},$$ whence we obtain the estimate $Sz(K_n, \ee)\leqslant CB(\mathcal{G}_n)$.    We now argue that if $n,m\in\nn$ are such that $\vartheta_m<2\ee$ and $m<n$, \begin{enumerate}[(i)]\item if $\xi=\omega^{\zeta+1}$ or $\xi=1$,  $Sz(K_n, \ee)\leqslant CB(\mathcal{G}_m)=(CB(\mathcal{G})-1)^m+1$,   \item if $\xi=\beta+\gamma$ for $\beta, \gamma$ as in the definition of $\xi$-well-constructed, $Sz(K_n, \ee)\leqslant \omega^\gamma+1$.  \end{enumerate}

Then for any $\ee>0$, if $m\in \nn$ is such that $\vartheta_m<2\ee$, we obtain the estimate $$Sz(K_n, \ee)\leqslant  CB(\mathcal{G}_m)<\omega^\xi$$ in the case $\xi=\omega^{\zeta+1}$ or $\xi=1$, and $$Sz(K_n, \ee) \leqslant \max\{\max_{0\leqslant n\leqslant m}CB(\mathcal{G}_m), \omega^\gamma+1\}<\omega^\xi$$ in the remaining case.   These estimates will finish the proof.

We will use the following fact: If $\mathcal{A}[\mathcal{B}]$ are regular families and $E\prec F\in \mathcal{A}[\mathcal{B}]$, then either $E\in \mathcal{A}'[\mathcal{B}]$ or $F\setminus E\in \mathcal{B}$.   Write $F=\cup_{i=1}^k F_i$, $F_1<\ldots <F_k$, $\varnothing \neq F_i\in \mathcal{B}$, $(\min F_i)_{i=1}^l\in \mathcal{A}$.  Then either $F\setminus E\subset F_k$ and therefore $F\setminus E$ lies in $\mathcal{B}$ by heredity, or there exists $0\leqslant l<k$ such that $E=\cup_{i=1}^l (E\cap F_i)$ and $E\cap F_i\neq \varnothing$ for each $1\leqslant i\leqslant l$.    In the second case, since $(\min (E\cap F_i))_{i=1}^l=(\min F_i)_{i=1}^l\prec (\min F_i)_{i=1}^k$, $E=\cup_{i=1}^l E\cap F_i$ witnesses the fact that $E\in \mathcal{A}'[\mathcal{B}]$.

Now in either of the cases $\xi=\omega^{\zeta+1}$ or $\xi=1$, we claim that for any ordinal $\eta$, $$s_\ee^\eta(K_n)\subset \{\vartheta_n \sum_{i\in F} e^*_i: F\in \mathcal{G}_m^\eta[\mathcal{G}_{n-m}]\},$$ which will give the result by taking $\eta=CB(\mathcal{G}_m)$.   In case (ii), we claim that $$s_\ee^\eta(K_n)\subset \{\vartheta_n \sum_{i\in F} e^*_i: F\in \mathcal{F}_n^\eta[\mathcal{G}_0]\},$$ which will give the desired conclusion taking $\eta=\omega^\gamma+1>CB(\mathcal{F}_n)$.    We prove these results by induction on $\eta$, with the $\eta=0$ case  being equality (noting that $\mathcal{G}_n=\mathcal{G}_m[\mathcal{G}_{n-m}]$) and the limit ordinal case being obvious.    Assume the result holds for some $\eta$.  In case (i), let $\mathcal{A}=\mathcal{G}_m^\eta$, $\mathcal{B}=\mathcal{G}_{n-m}$.  In case $(ii)$, let $\mathcal{A}=\mathcal{F}_n^\eta$ and $\mathcal{B}=\mathcal{G}_0$.    We must show that $$s_\ee(\{\vartheta_n \sum_{i\in F} e^*_i: F\in \mathcal{A}[\mathcal{B}]\})\subset \{\vartheta_n\sum_{i\in F}e^*_i: F\in \mathcal{A}'[\mathcal{B}]\},$$ which will complete the induction and the proof.  Now if $\vartheta_n\sum_{i\in E}e^*_i\in s_\ee(\{\vartheta^n \sum_{i\in F} e^*_i: F\in \mathcal{A}[\mathcal{B}]\})$, there exists $F\in \mathcal{A}[\mathcal{B}]$ with $E\prec F$ such that $$\ee/2<\|\vartheta_n \sum_{i\in F\setminus E} e^*_i\|.$$   It suffices to show that $E\in \mathcal{A}'[\mathcal{B}]$.  If $E\notin \mathcal{A}'[\mathcal{B}]$, then $F\setminus E\in \mathcal{B}$.  In case (i), $\vartheta^{n-m}\sum_{i\in F\setminus E} e^*_i\in K_{n-m}\subset B_{X^*}$, whence $$\ee/2 < \vartheta^m \|\vartheta^{n-m}\sum_{i\in F\setminus E} e^*_i\| \leqslant \vartheta^m,$$ a contradiction.  In case (ii), $\sum_{i\in F\setminus E}e^*_i\in K_0\subset B_{X^*}$, whence $$\ee/2 <\vartheta_n  \|\sum_{i\in F\setminus E}e^*_i\|\leqslant \vartheta_n <\vartheta_m,$$ a contradiction.

\end{proof}

\begin{proposition} Suppose $X$ is a Banach space, $Z$ is a subspace of $X$ with $\dim X/Z<\infty$, $K\subset X^*$ is weak$^*$-compact, and $x^*\in s_\ee(K)$.  Then for any $0<\delta<\ee/4$ and any weak$^*$-neighborhood $v$ of $x^*$, there exist $y^*\in K$ and $z\in B_Z$ such that $\text{\emph{Re\ }}y^*(z)>\delta$. 

\label{stang}

\end{proposition}

\begin{proof} Fix $R>1$ such that $K\subset RB_{X^*}$.   We may fix a net $(x^*_\lambda)\subset K$ converging weak$^*$ to $x^*$ and such that $\|x^*_\lambda-x^*\|>\ee/2$ for all $\lambda$. For each $\lambda$, we may fix $x_\lambda\in B_X$ such that $\text{Re\ }(x^*_\lambda-x^*)(x_\lambda)$.   Fix $\eta>0$ such that $\delta+3R\eta <\ee/4$.  After passing to a subnet, we may assume that $|x^*(x_{\lambda_2}-x_{\lambda_1})|<\eta$ and $\|x_{\lambda_2}-x_{\lambda_1}\|_{X/Z}<\eta$ for all $\lambda_1, \lambda_2$.       Now fix any $\lambda_1$ and then choose $\lambda_2$ such that $|(x^*_{\lambda_2}-x^*)(x_{\lambda_1})|<\eta$ and such that $x^*_{\lambda_2}\in v$.      We may now fix $z\in B_Z$ such that $\|\frac{x_{\lambda_2}-x_{\lambda_1}}{2}-z\|<2\eta$ and let $y^*=x^*_{\lambda_2}\in K$.  Now note that \begin{align*} \text{Re\ }y^*(z) & > \text{Re\ }x^*_{\lambda_2}\Bigl(\frac{x_{\lambda_2}-x_{\lambda_1}}{2}\Bigr) - R\eta \\ & > \text{Re\ }(x^*_{\lambda_2}-x^*)\Bigl(\frac{x_{\lambda_2}-x_{\lambda_1}}{2}\Bigr) - 2R\eta \\ & > \text{Re\ }(x^*_{\lambda_2}-x^*)(x_{\lambda_2}/2) - 3R\eta > \ee/4-3R\eta>\delta. \end{align*}

\end{proof}

The following can be compared to Proposition $5$ of \cite{OSZ}.

\begin{corollary} Suppose $\mathcal{G}$ is a regular family with $CB(\mathcal{G})=\xi+1$. If $\textsf{\emph{F}}$ is any FMD for $X$, $K\subset X^*$ is weak$^*$-compact, and $x^*\in s_\ee^\xi(K)$, then for any $0<\delta<\ee/4$, there exist a collection $(x_t)_{t\in \mathcal{G}\setminus MAX(\mathcal{G})}\subset B_X$ and a collection $(x^*_t)_{t\in \mathcal{G}}\subset K$ such that $x^*_\varnothing=x^*$, $x_E\in \text{\emph{span}}\{F_j: j> \max E\}$, and  if $\varnothing\prec E\preceq F\in \mathcal{G}$, then $\text{\emph{Re\ }}x^*_F(x_E)>\delta$.

\label{abandoned}
\end{corollary}

\begin{proof} Define $\mu:\mathcal{G}\to [0, \xi]$ by letting $\mu(E)=\max\{\zeta: E\in \mathcal{G}^\zeta\}$.  We will define $(x_E)_{E\in \mathcal{G}\setminus MAX(\mathcal{G})}$ and $(x^*_E)_{E\in \mathcal{G}}$ recursively to have each of the properties mentioned in the corollary, and to have the property that for each $E\in \mathcal{G}$, $x^*_E\in s^{\mu(E)}_\ee(K)$.  We let $x^*_\varnothing=x^*$.

Now suppose that $\varnothing\prec E\in \mathcal{G}$,  $x^*_{E^-}\in s^{\mu(E^-)}_\ee(K)$, and $x_G\in \text{span}(F_j:j\in\nn)$ have been defined for each $\varnothing\prec G\prec E$.  If $E^-=\varnothing$, let $v=X^*$, and otherwise let $v=\{y^*\in Y^*:(\forall \varnothing\prec G\prec E)(\text{Re\ }y^*(x_G)>\delta)\}$, which is a weak$^*$-neighborhood of $x^*_{E^-}$.   Let $p=\max E$ and $Z=[F_j:j> p]$.  Note that  $\mu(E)<\mu(E^-)$. Since $x^*_{E^-}\in s^{\mu(E^-)}_\ee(K)\subset s^{\mu(E)+1}_\ee(K)=s_\ee(s_\ee^{\mu(E)}(K))$, Proposition \ref{stang} yields the existence of $z\in B_Z$ and $x^*_E\in s^{\mu(E)}_\ee(K)\cap v$ such that $\text{Re\ }x^*_E(z)>\delta$.   By density of $\text{span}\{F_j:j>p\}$ in $[F_j:j>p]$, we may fix $x_E\in B_X\cap \text{span}\{F_j:j>p\}$ such that $\text{Re\ }x^*_E(x_E)>\delta$. This completes the recursive construction, and the collections $(x^*_E)_{E\in \mathcal{G}}$, $(x_E)_{E\in \mathcal{G}\setminus MAX(\mathcal{G})}$ are easily seen to satisfy the conclusions.

\end{proof}

Now for a Banach space $X$, an FMD $\textsf{F}$ of $X$, $K\subset X^*$ weak$^*$-compact, and $\ee>0$, let $\mathcal{H}(X,\textsf{F},K, \ee)=\varnothing$ if $K=\varnothing$, and otherwise let $\mathcal{H}(X,\textsf{F},K, \ee)$ denote the collection consisting of $\varnothing$ together with all $(k_i)_{i=1}^n\in [\nn]^{<\nn}$ such that (with $k_0=0$), there exist $x^*\in K$ and $(u_i)_{i=1}^n\in B_X\cap \prod_{i=1}^n [F_j: k_{i-1}<j\leqslant k_i]$ such that $|x^*(u_i)|\geqslant \ee$ for all $1\leqslant i\leqslant n$ (equivalently, such that $\text{Re\ }x^*(u_i)\geqslant \ee$ for all $1\leqslant i\leqslant n$).

\begin{lemma} For any Banach space $X$, any weak$^*$-compact subset $K$ of $X^*$, any $K$-shrinking FMD $\textsf{\emph{F}}$ of $X$,  $0<\delta<\ee$ and any ordinal $\xi$, $$\mathcal{H}(X, \textsf{\emph{F}}, K, \ee)^{2\xi}\subset \mathcal{H}(X, \textsf{\emph{F}}, s^\xi_\delta(K), \ee).$$   In particular, if $\mathcal{H}(X, \textsf{\emph{f}}, K, \ee)^{2\xi}\neq \varnothing$, then $s^\xi_\delta(K)\neq \varnothing$.

\label{geoff}
\end{lemma}

\begin{proof} In the proof, we will repeatedly use the fact that for a weak$^*$-compact subset $L$ of $X^*$, $\mathcal{H}(X, \textsf{F}, L, \ee)\neq \varnothing$ if and only if $L\neq \varnothing$ if and only if $\varnothing\in \mathcal{H}(X, \textsf{F}, L, \ee)$.

We induct on $\xi$.    The $\xi=0$ case is trivial.

Assume $\xi$ is a limit ordinal and the result holds for all $\zeta<\xi$.  Note that by the properties of ordinals, $2\xi=\xi$ and $2\zeta<\xi$ for every $\zeta<\xi$.    Suppose that  for some $n\in \nn\cup \{0\}$, $$(k_i)_{i=1}^n\in \mathcal{H}(X, \textsf{F}, K, \ee)^{2\xi}=\mathcal{H}(X, \textsf{F}, K, \ee)^{\xi}=\bigcap_{\zeta<\xi} \mathcal{H}(X, \textsf{F}, K, \ee)^{2\zeta}.$$  Here, if $n=0$, $(k_i)_{i=1}^0$ denotes the empty sequence by convention.   

If $n>0$, then for every $\zeta<\xi$, we may fix $x^*_\zeta\in s^\zeta_\delta(K)$ and $(u^\zeta_i)_{i=1}^n\in B_X\cap \prod_{i=1}^n [F_j: k_{i-1}<j\leqslant k_j]=:C$ such that for each $1\leqslant i\leqslant n$, $\text{Re\ }x^*_\zeta(u^\zeta_i)\geqslant \ee$.  If $C$ is endowed with the product of the norm topology and $K$ is endowed with its weak$^*$-topology, by compactness of $C\times K$, we may fix $$(u_1, \ldots, u_n, x^*)\in \bigcap_{\eta<\xi}\overline{\{(u^\zeta_1, \ldots, u^\zeta_n, x^*_\zeta): \eta<\zeta<\xi \}}.$$   Obviously $x^*\in s^\xi_\delta(K)$ and $\text{Re\ }x^*(u_i)\geqslant \ee$ for each $1\leqslant i\leqslant n$, witnessing that $(k_i)_{i=1}^n\in \mathcal{H}(X, \textsf{F}, s^\xi_\delta(K), \ee)$.    If $n=0$, we omit reference to $u^\zeta_i$, $u_i$, and $C$ in the previous argument and use the fact at the beginning of the proof to deduce that $(k_i)_{i=1}^0=\varnothing\in \mathcal{H}(X, \textsf{F}, s^\xi_\delta(K), \ee)$.

Assume the result holds for $\xi$ and $ (k_i)_{i=1}^n\in \mathcal{H}(X, \textsf{F}, K, \ee)^{2(\xi+1)}= \mathcal{H}(X, \textsf{F}, K, \ee)^{2\xi+2}$.  First suppose $n>0$.  Then there exists a sequence $l_1<m_1<l_2<m_2<\ldots$ such that for all $t\in \nn$, $(k_i)_{i=1}^n\smallfrown (l_t, m_t)\in \mathcal{H}(X, \textsf{F}, K, \ee)^{2\xi}$. By the inductive hypothesis, for each $t\in \nn$, there exist $x^*_t\in s^\xi_\delta(K)\subset K$,  $(u^t_i)_{i=1}^n\in B_X\cap \prod_{i=1}^n [F_j: k_{i-1}<j\leqslant k_i]=:C$, and $v_t\in B_X\cap [E_j:l_t<j\leqslant m_t]$ such that $\text{Re\ }x^*_t(v_t)\geqslant \ee$ and for each $1\leqslant i\leqslant n$, $\text{Re\ }x^*_t(v_t)\geqslant \ee$.   We may pass to a subsequence and use the sequential compactness of $C$ with the product of its norm topology and $K$ with its weak$^*$-topology to assume $u^t_i\to u_i$  and $x^*_t\underset{\text{weak}^*}{\to} x^*$. Obviously $\text{Re\ }x^*(u_i)\geqslant \ee$ for all $1\leqslant i\leqslant n$.    Since $\textsf{F}$ is $K$-shrinking and $v_t\in B_X\cap [E_j:l_t<j\leqslant m_t]$, $$\underset{t}{\lim\inf} \|x^*_t-x^*\| \geqslant \underset{t}{\lim\inf} \text{Re\ }(x^*_t-x^*)(v_t)\geqslant \ee>\delta.$$   Since $(x^*_t)_{t=1}^\infty \subset s^\xi_\delta(K)$, $x^*\in s^{\xi+1}_\delta(K)$.    This yields that $(k_i)_{i=1}^n\in \mathcal{H}(X, \textsf{F}, s^{\xi+1}_\delta(K), \ee)$.  If $n=0$, we omit reference to $u^t_i$ and $u_i$ in the previous argument and use the remark at the beginning of the proof to deduce that $(k_i)_{i=1}^0=\varnothing\in \mathcal{H}(K, \textsf{F}, s^{\xi+1}_\delta, \ee)$.

\end{proof}

\begin{corollary}  If $X$ is a Banach space, $K\subset X^*$ is weak$^*$-compact,  $\textsf{\emph{F}}$ is a $K$-shrinking FDD for $X$, then for any $\ee>0$, $$Sz(K, 5\ee)\leqslant CB(\mathcal{H}(X, \textsf{\emph{F}}, K, \ee)) \leqslant 2Sz(K, \ee/2).$$

In particular, if $K$ is convex and not norm compact, $$Sz(K)=\sup_{\ee>0} CB(\mathcal{H}(X, \textsf{\emph{F}}, K, \ee)).$$   

\label{bad}

\end{corollary}

\begin{proof} The proof of the first part follows from Corollary \ref{abandoned} and  Lemma \ref{geoff}. The second part follows from the fact that if $K$ is convex and not norm compact, either $Sz(K)=\infty=\sup_{\ee>0} CB(\mathcal{H}(X, \textsf{F}, K, \ee))$, and otherwise $Sz(K)=\omega^\xi$ for some $0<\xi<\omega_1$. In this case, for each $\ee>0$, $2Sz(K,\ee/2)<\omega^\xi$, so $$\omega^\xi =\sup_{\ee>0} Sz(K, 5\ee) \leqslant \sup_{\ee>0} CB(\mathcal{H}, \textsf{F}, K, \ee) \leqslant \sup_{\ee>0} 2Sz(K, \ee/2)=\omega^\xi.$$

\end{proof}

We next prove a generalization of a result of Schlumprecht, which was shown in the case $K=B_{X^*}$.

\begin{lemma} Suppose $X$ is   a Banach space $X$, $K\subset X^*$ is weak$^*$-compact, $\textsf{\emph{F}}$ is a $K$-shrinking FMD for $X$, and $0<\xi<\omega_1$. Then $Sz(K)\leqslant \omega^\xi$ if and only if for any $\ee>0$ and any $L\in [\nn]$, there exists $M\in [L]$ such that $$\sup \Bigl\{\Bigl\langle\sum_{j=1}^\infty \mathbb{S}^\xi_{N,1}(n_j)e_j\Bigr\rangle_{X, \textsf{\emph{F}}, K, N}:N\in [M] \Bigr\} \leqslant \ee.$$

\label{sprites}

\end{lemma}

\begin{proof} Throughout the proof, for ease of notation, let $\langle\cdot\rangle_M=\langle \cdot\rangle_{X, \textsf{F}, K, M}$.  First suppose that $Sz(K)>\omega^\xi$.  Fix $\ee$ such that $Sz(K, 15\ee)>\omega^\xi$.    Then by Corollary \ref{abandoned}, there exist  collections $(x_E)_{E\in \mathcal{S}_\xi\setminus \{\varnothing\}}\subset B_X$ and $(x^*_E)_{E\in MAX(\mathcal{S}_\xi)}\subset K$ such that for every $\varnothing\prec E\preceq F\in MAX(\mathcal{S}_\xi)$, $\text{Re\ }x^*_F(x_E)\geqslant 3\ee$, and such that $u_E\in \text{span}\{F_j: j> \max E\}$.  Seeking a contradiction, suppose that $M=(m_i)_{i=1}^\infty\in [\nn]$ is such that $$\Bigl\langle \sum_{j=1}^\infty \mathbb{S}^\xi_{N, 1}(n_j)e_j\Bigr\rangle_N \leqslant \ee$$ for all $N\in[M]$.   Fix $3\leqslant n_1$. Assuming that $n_1<\ldots<n_k$ have been chosen, if $(n_1, \ldots, n_k)\in \mathcal{S}_\xi$, fix $n_k<n_{k+1}\in M$ such that $x_{(n_1, \ldots, n_k)}\in \text{span}\{F_j: j<n_{k+1}\}$.  If $(n_1, \ldots, n_k)\notin \mathcal{S}_\xi$, fix $n_k<n_{k+1}\in M$ arbitrary.  By compactness of $\mathcal{S}_\xi$ together with the fact that for each $\varnothing \neq E\in \mathcal{S}_\xi\setminus MAX(\mathcal{S}_\xi)$, $E\cup (1+\max E)\in \mathcal{S}_\xi$,  there exists $k\in \nn$ such that $G=(n_1, \ldots, n_k)\in MAX(\mathcal{S}_\xi)$.    Let $n_0=0$ and $x_1=0\in [F_j:  n_0<j\leqslant n_1]$.    For $1<i\leqslant k$, let $x_i=x_{(n_1, \ldots, n_{i-1})}\in [F_j: n_{i-1}<j\leqslant n_i]$.    Then $$\ee\geqslant \Bigl\langle\sum_{i=1}^\infty \mathbb{S}^\xi_{N,1}(n_j)e_j\Bigr\rangle_N \geqslant \text{Re\ }x^*_G(\sum_{j=1}^k\mathbb{S}^\xi_{N,1}(n_j)x_j) \geqslant 3\ee(1-\mathbb{S}^\xi_{N,1}(n_1))\geqslant 3\ee(2/3)=2\ee,$$ a contradiction.

Now suppose that $Sz(K)\leqslant \omega^\xi$.  Fix $\ee>0$ and let $\mathcal{V}$ denote the set of $M\in [\nn]$ such that $$\Bigl\langle\sum_{j=1}^\infty \mathbb{S}^\xi_{M,1}(n_j)e_j\Bigr\rangle_M \leqslant \ee.$$  By the permanence properties of the measures $\mathbb{S}^\xi_{M,i}$, it follows that $\mathcal{V}$ is closed.   Then there exists $M\in[\nn]$ such that either $[M]\cap \mathcal{V}=\varnothing$ or $[M]\subset \mathcal{V}$.  We will show that $[M]\subset \mathcal{V}$, which will finish the proof.  Seeking a contradiction, assume $[M]\cap \mathcal{V}=\varnothing$.   For each $F=(n_i)_{i=1}^k\in MAX(\mathcal{S}_\xi)\cap [M]^{<\nn}$, we may fix $N_F\in [M]$ such that $F$ is an initial segment of $N_F$.  Then since $$\ee<\Bigl\langle \sum_{j=1}^\infty \mathbb{S}^\xi_{N_F, 1}(n_j)e_j\Bigr\rangle_{N_F},$$ there exist $x^*_F\in K$ and $(x_i^F)_{i=1}^k\in B_X^k\cap \prod_{i=1}^k [F_j: n_{i-1}<j\leqslant n_i]$ such that $\text{Re\ }x^*_F(\sum_{j=1}^k \mathbb{S}^\xi_{N_F,1}(n_j) x^F_i) \geqslant \ee$.    Define $f:\mathcal{S}_\xi\bowtie M\to \rr$ as follows: If $F=(n_i)_{i=1}^k\in MAX(\mathcal{S}_\xi)\cap [M]^{<\nn}$, let $f(n_i, F)=\text{Re\ }x^*_F(x_i^F)$.   Then by Proposition \ref{CN}, there exists $P\in [\nn]$ such that for any $E\in \mathcal{S}_\xi$, there exists $P(E)\subset F\in MAX(\mathcal{S}_\xi)\cap [M]^{<\nn}$ such that for each $j\in P(E)$, $f(j, F)\geqslant \ee/2$.    From this it easily follows that $\mathcal{S}_\xi(P)\subset \mathcal{H}(X, \textsf{F}, K, \ee/2)$, and $CB(\mathcal{H}(X, \textsf{F}, K, \ee/2))\geqslant CB(\mathcal{S}_\xi)=\omega^\xi+1$.  But since $Sz(K, \ee/3)<\omega^\xi$, $$CB(\mathcal{H}(X, \textsf{F}, K, \ee/2)) \leqslant 2 Sz(K, \ee/2)<\omega^\xi,$$ a contradiction.

\end{proof}

\begin{corollary}  Suppose $E$ is a sequence space the canonical basis of which is shrinking.  Suppose that $X$ is a Banach space with bimonotone FDD $\textsf{\emph{F}}$.  Then $\textsf{\emph{F}}$ is shrinking in $X^E_\wedge(\textsf{\emph{F}})$ and  $$Sz(X^E_\wedge(\textsf{\emph{F}}))\leqslant Sz(E).$$

\label{7u}
\end{corollary}

\begin{proof} First, we recall the following easy fact. For any Banach space $Z$ with FDD $\textsf{G}$,  then $\textsf{G}$ is a shrinking FDD for $Z$ if and only if for every $\ee>0$, $CB(\mathcal{H}(Z, \textsf{G}, B_{Z^*}, \ee))<\omega_1$.  Since $$\langle\cdot\rangle_{X^E_\wedge(\textsf{F}), \textsf{F}, B_{X^E_\wedge(\emph{F})^*}, M}\leqslant 2 \langle \cdot\rangle_{E, \textsf{E}, B_{E^*}, M}$$ for all $M\in[\nn]$, $$CB(\mathcal{H}(X^E_\wedge(\textsf{F}), \textsf{F}, B_{X^E_\wedge(\textsf{F})^*}, \ee) ) \leqslant CB(\mathcal{H}(E, \textsf{E}, B_{E^*}, \ee/2))$$ for all $\ee>0$.  Since $\textsf{E}$ is shrinking in $E$, the latter value is countable for each $\ee>0$, as is the former. From this it follows that $\textsf{F}$ is shrinking in $X^E_\wedge(\textsf{F})$.

Since $\dim E=\infty$ and $E^*$ is separable, $Sz(E)=\omega^\xi$ for some $0<\xi<\omega_1$.   Since $$\langle \cdot\rangle_{X^E_\wedge(\textsf{F}), \textsf{F}, B_{(X^E_\wedge(\textsf{F}))^*}, N}\leqslant 2\langle \cdot\rangle_{E, \textsf{E}, B_{E^*}, N}$$ for any $N\in [\nn]$, an appeal to Lemma \ref{sprites} gives the result.

\end{proof}

\begin{corollary} Fix $\xi<\omega_1$ and $p,q$ with $1\leqslant q<\infty$ and $1/p+1/q=1$.   Suppose $X$ is a Banach space, $K\subset X^*$, and $\textsf{\emph{F}}$ is a $K$-shrinking FMD for $X$.    Then $\textsf{\emph{p}}_\xi(K)\leqslant q$ if and only if for each $1<r<p$, there exist a blocking $\textsf{\emph{G}}$ of $\textsf{\emph{F}}$, a sequence $(m_n)_{n=1}^\infty \in [\nn]$,  and a constant $C\geqslant 0$ such that for each $0=r_0<r_1<\ldots$, each $(u_i)_{i=1}^\infty\in B_X^{<\nn}\cap \prod_{i=1}^\infty [G_j: r_{i-1}<j\leqslant r_i]$, and each $(a_i)_{i=1}^\infty\in c_{00}$, $$r_K(\sum_{i=1}^\infty a_i u_i) \leqslant C\|\sum_{i=1}^\infty a_i e_{m_{r_i}}\|_{\xi,r}.$$

\label{nowitstimetoreallybaernnotice}
\end{corollary}

\begin{proof} First suppose that $\textsf{p}_\xi(K)\leqslant q$. Fix $1<r<\alpha<p$ and let $1=1/r+1/s=1/\alpha+1/\beta$.    Fix $m\in\nn$ such that $2^m>\sup_{x^*\in K}\|x^*\|$.  Since $\textsf{p}_\xi(K)\leqslant q$, there exists $l\in \nn$ such that for all $n\in\nn$, $$2Sz_\xi(K, 2^m/2^{1+n/\beta})< l 2^n.$$   Recursively select $M_1\supset M_2\supset \ldots$ such that for each $n\in\nn$, either $$\mathcal{H}(X, \textsf{F}, K, 2^m/2^{n/\beta}) \cap [M_n]^{<\nn}\subset \mathcal{A}_{l2^n}[\mathcal{S}_\xi]$$ or $$\mathcal{A}_{l2^n}[\mathcal{S}_\xi]\cap [M_n]^{<\nn}\subset \mathcal{H}(X, \textsf{F}, K, 2^m/2^{n/\beta}).$$  But since $$CB(\mathcal{H}(X, \textsf{F}, K, 2^m/2^{n/\beta}))<\omega^\xi l2^n+1= CB(\mathcal{A}_{l2^n}[\mathcal{S}_\xi]),$$ the first inclusion must hold.  Now fix $m_1<m_2<\ldots$, $m_n\in M_n$.   For each $n\in\nn$, let $G_n=[F_j: m_{n-1}<j\leqslant m_n]$ and let $\textsf{G}=(G_n)_{n=1}^\infty$.   Let $$C=\sum_{n=1}^\infty \frac{n2^m}{2^{(n-1)/\beta}} + \frac{2^m (l2^n)^{1/s}}{2^{(n-1)/\beta}} <\infty.$$   Now fix $0=r_0<r_1<\ldots$, $(u_i)_{i=1}^\infty\in B_X^\nn\cap \prod_{i=1}^\infty [G_j: r_{i-1}<j\leqslant r_i]$, $(a_i)_{i=1}^\infty\in c_{00}$, and $x^*\in K$ such that $$r_K(\sum_{i=1}^\infty a_i u_i)=\text{Re\ }x^*(\sum_{i=1}^\infty a_i u_i).$$   For each $n\in\nn$, let $$B_n=\{i<n: |x^*(u_i)|\in (2^m/2^{n/\beta}, 2^m/2^{(n-1)/\beta}]\}$$ and $$C_n=\{i\geqslant n: |x^*(u_i)|\in (2^m/2^{n/\beta}, 2^m/2^{(n-1)/\beta}]\}.$$   Then for any $n\in\nn$, $$\text{Re\ }x^*(\sum_{i\in B_n} a_iu_i)\leqslant \frac{2^m}{2^{(n-1)/\beta}}\sum_{i\in B_n}|a_i|\leqslant \frac{n2^m}{2^{(n-1)/\beta}}\|\sum_{i=1}^\infty a_i e_{m_{r_i}}\|_{\xi,r}.$$   For any $n\in\nn$, $$(m_{r_i})_{i\in C_n}\in \mathcal{H}(X, \textsf{F}, K, 2^m/2^{n/\beta})\cap [M_n]^{<\nn}\subset \mathcal{A}_{l2^n}[\mathcal{S}_\xi].$$  Now write $(m_{r_i})_{i\in C_n}= \cup_{i=1}^k E_i$, $k\leqslant l2^n$, $\varnothing\neq E_i\in \mathcal{S}_\xi$.   Then \begin{align*} \text{Re\ }x^*(\sum_{i\in C_n}a_i u_i) & \leqslant \frac{2^m}{2^{(n-1)/\beta}} \sum_{i\in C_n}|a_i| = \frac{2^m}{2^{(n-1)/\beta}}\sum_{j=1}^k \sum_{i\in E_j} |a_i| \\ & \leqslant \frac{2^m (l2^n)^{1/s}}{2^{(n-1)/\beta}}\bigl(\sum_{j=1}^k \bigl(\sum_{i\in E_j}|a_i|  \bigr)^r\bigr)^{1/r} \\ & \leqslant \frac{2^m (l2^n)^{1/s}}{2^{(n-1)/\beta}} \|\sum_{i=1}^\infty a_i e_{m_{r_i}}\|_{\xi,r}. \end{align*}  Then $$r_K(\sum_{i=1}^\infty a_iu_i) \leqslant \sum_{n=1}^\infty \text{Re\ } x^*( \sum_{i\in B_n\cup C_n}a_iu_i) \leqslant C\|\sum_{i=1}^\infty a_i e_{m_{r_i}}\|_{\xi,r}.$$

Now suppose that for every $1<r<p$, the blocking $\textsf{G}$, the sequence $(m_n)_{n=1}^\infty$, and the constant $C$ exist.    Now fix $1<r<p$, let $1/r+1/s=1$, and let $\textsf{G}$, $(m_n)_{n=1}^\infty$, and $C$ be as in the statement. By replacing $C$ with a larger value if necessary, we may assume $C\geqslant 1$.   For each $0<\ee\leqslant 1$, let $$\mathcal{B}_\ee=\{E\in [\nn]^{<\nn}: (\forall x\in \text{co}(e_i: i\in E))(\|x\|_{\xi,r}\geqslant \ee)\}$$ and note that $CB(\mathcal{B}_\ee) \leqslant \omega^\xi \lfloor \ee^{-s}\rfloor+1$ for all $\ee\in (0,1]$ by Lemma \ref{door}.  Let $\mathcal{C}_\ee=\{E: M(E)\in \mathcal{B}_\ee\}$ and note that $\mathcal{C}_\ee$ is homeomorphic to $\mathcal{B}_\ee$, whence $CB(\mathcal{C}_\ee)\leqslant \omega^\xi \lfloor \ee^{-s}\rfloor +1$ for all $\ee\in (0,1]$.    Now suppose that $(r_i)_{i=1}^l\in \mathcal{H}(X, \textsf{G}, K, \ee)$.      Fix any $r_l<r_{l+1}<r_{l+2}<\ldots$.   By definition, there exist $(u_i)_{i=1}^l\in B_X^l\cap \prod_{i=1}^l [G_j: r_{i-1}<j\leqslant r_i]$ and $x^*\in K$ such that $\text{Re\ }x^*(u_i)\geqslant \ee$ for each $1\leqslant i\leqslant l$.   Let $u_i=0$ for all $i>l$.     Fix non-negative scalars $(a_i)_{i=1}^l$ summing to $1$ and note that $$C\|\sum_{i=1}^l a_i e_{m_{r_i}}\|\geqslant r_K(\sum_{i=1}^l a_iu_i)\geqslant \text{Re\ }x^*(\sum_{i=1}^l a_i u_i) \geqslant \ee,$$ whence $(m_{r_i})_{i=1}^l\in \mathcal{B}_{\ee/C}$ and $(r_i)_{i=1}^l\in \mathcal{C}_{\ee/C}$.   We have shown that $$\mathcal{H}(X, \textsf{G}, K, \ee)\subset \mathcal{C}_{\ee/C},$$ whence $$Sz(K, 5\ee)\leqslant CB(\mathcal{H}(X, \textsf{G}, K, \ee)) \leqslant CB(\mathcal{C}_{\ee/C}) \leqslant \omega^\xi \lfloor (\ee/C)^{-s}\rfloor+1.$$  From this it easily follows that there exists a constant $D$ such that for any $0<\ee<1$, $Sz_\xi(K, \ee)\leqslant D/\ee^s$, and $\textsf{p}_\xi(K)\leqslant s$.     Since $1<r<p$ was arbitrary, $\textsf{p}_\xi(K)\leqslant q$.

\end{proof}

We next collect an embedding theorem which combines results from \cite{first} and \cite{second}.

\begin{theorem} Fix $\xi<\omega_1$. \begin{enumerate}[(i)]\item If $X$ is a Banach space with separable dual and $Sz(X)\leqslant \omega^\xi$, then there exists a Banach space $W$ with bimonotone FDD $\textsf{\emph{F}}$ such that $X$ is isomorphic to both a subspace and a quotient of $W^{X_\xi}_\wedge(\textsf{\emph{F}})$.  \item If $X$ is a Banach space, $1/p+1/q=1$, and $\textsf{\emph{p}}_\xi(X)<q$, then there exists a Banach space $W$ with bimonotone FDD $\textsf{\emph{F}}$ such that $X$ is isomorphic to both a subspace and a quotient of $W^{X_{\xi,p}}_\wedge(\textsf{\emph{F}})$. \end{enumerate}

\label{ute}

\end{theorem} 

\begin{proof} $(i)$ Let $\mathcal{E}=\{(n_1, \ldots, n_{2k}): k\in\nn, n_1<\ldots <n_{2k}\}$.     We first remark that it was shown in \cite{first} that if $Sz(X)\leqslant \omega^\xi$, then there exists a constant $C$ such that for any collection $(x_E)_{E\in \mathcal{E}}\subset B_X$ such that for each $n_1<\ldots <n_{2k-1}$, $(x_{(n_1, \ldots, n_{2k-1}, n_k)})_{n_k>n_{2k-1}}$ is weakly null, there exist $n_1<n_2<\ldots$ such that for any $(a_i)_{i=1}^\infty \in c_{00}$, $$\|\sum_{i=1}^\infty a_i x_{(n_1, \ldots, n_{2i})}\|\leqslant C\|\sum_{i=1}^\infty a_i e_{n_{2i-1}}\|_{X_\xi}.$$   From the main embedding theorem of \cite{first}, since the canonical basis of $X_\xi$ is shrinking and has properties $R$, $S$, and $T$, there exist Banach spaces $U,V$ with bimonotone FDDs $\textsf{G}$ and $\textsf{H}$ such that $X$ is isomorphic to a subspace of $\widehat{U}$ and to a quotient of $\widehat{V}$, where the norm on $\widehat{U}$ is given by $$\|u\|_{\widehat{U}}= \sup\Bigl\{\|\sum_{i=1}^\infty \|I^\textsf{G}_i u\|_U e_{\min I_i}\|_{X_\xi}: I_1<I_2<\ldots, I_i\text{\ an interval}\Bigr\}$$ and the norm of $\widehat{V}$ is given by $$\|v\|_{\widehat{V}}= \sup\Bigl\{\|\sum_{i=1}^\infty \|I^\textsf{G}_i u\|_V e_{\min I_i}\|_{X_\xi}: I_1<I_2<\ldots, I_i\text{\ an interval}\Bigr\}.$$  Since $X_\xi$ has properties $S$ and $T$, the norms of $\widehat{U}$ and $\widehat{V}$ are equivalent to $\|\cdot\|_{U^{X_\xi}_\wedge(\textsf{G})}$ and $\|\cdot\|_{V^{X_\xi}_\wedge(\textsf{H})}$, respectively.   Let $F_n=G_n\oplus_\infty H_n$ and $W=U\oplus_\infty V$.  Then $X$ is isomorphic to a subspace and a quotient of $W^{X_\xi}_\wedge(\textsf{F})$.

$(ii)$ This is similar to $(i)$.   We only need to show that if $X$ is a Banach space with separable dual and $\textsf{p}_\xi(X)<q$, then there exists constant $C'$ such that   for any $(x_E)_{E\in \mathcal{E}}\subset B_X$ such that for each $n_1<\ldots <n_{2k-1}$, $(x_{(n_1, \ldots, n_{2k-1}, n_k)})_{n_k>n_{2k-1}}$ is weakly null, there exist $n_1<n_2<\ldots$ such that for any $(a_i)_{i=1}^\infty \in c_{00}$, $$\|\sum_{i=1}^\infty a_i x_{(n_1, \ldots, n_{2i})}\|\leqslant C'\|\sum_{i=1}^\infty a_i e_{n_{2i-1}}\|_{X_{\xi,p}}.$$  We note that, as shown in \cite{second},  the canonical basis of $X_{\xi,p}$ is shrinking, and $X_{\xi,p}$ has properties $R$, $S$, and $T$,  so the  main embedding theorem from \cite{first} applies.  In order to find the indicated constant $C'$, we note that by Corollary \ref{nowitstimetoreallybaernnotice}, there exist an FMD $\textsf{I}=(I_n)_{n=1}^\infty$ for $X$ and $m_1<m_2<\ldots$ such that for any $0=r_0<r_1<\ldots$, any $(u_i)_{i=1}^\infty B_X^\nn\cap \prod_{i=1}^\infty [I_j: r_{i-1}<j\leqslant r_i]$, and any $(a_i)_{i=1}^\infty \in c_{00}$, $$\|\sum_{i=1}^\infty a_iu_i\|\leqslant C\|\sum_{i=1}^\infty a_i e_{m_{r_i}}\|_{X_{\xi,p}}.$$   Now note that, since $X_{\xi,p}$ has property $S$, there exists a constant $D$ such that for any $s_1<t_1<s_2<t_2<\ldots$ and any $(a_i)_{i=1}^\infty \in c_{00}$, $$\|\sum_{i=1}^\infty a_i e_{t_i}\|_{\xi,p} \leqslant D\|\sum_{i=1}^\infty a_i e_{s_i}\|_{\xi,p}.$$      Let $C'=CD+1$.  Fix a sequence of positive numbers $(\ee_n)_{n=1}^\infty$ such that $\sum_{n=1}^\infty\ee_n=1$ and suppose $(x_E)_{E\in \mathcal{E}}$ is as above.  Let us recursively select $n_1<n_2<\ldots$, $t_1<t_2<\ldots$, and $u_i\in B_X$ such that \begin{enumerate}[(i)]\item $\|x_{(n_1, \ldots, n_{2i})}-u_i\|_X<\ee_i$, \item $u_i\in [I_j: t_{i-1}<j\leqslant t_i]$, \item $n_{2i-1}<m_{t_i}<n_{2i+1}$. \end{enumerate}

We may fix $n_1=1$, $n_2=2$, $u_1\in B_X\cap \text{span}\{I_j: j\in\nn\}$ such that $\|x_{(n_1, n_2)}-u_1\|<\ee_1$, and $t_1\in \nn$ such that $u_1\in \text{span}\{I_j: j\leqslant t_1\}$.  Now assume that $n_1<\ldots <n_{2i}$, $t_1<\ldots < t_i$, $u_1, \ldots, u_i$ have been chosen.   Fix $n_{2i+1}>m_{t_i}$.  Then choose $n_{2i+2}>n_{2i+1}$ such that $$d(x_{(n_1, \ldots, n_{2i+2})}, B_X\cap [I_j: j>t_i])<\ee_{i+1},$$ $u_{i+1}\in B_X\cap \text{span}\{I_j: j>t_i\}$ such that $\|x_{(n_1, \ldots, n_{2i+2})}-u_{i+1}\|<\ee_{i+1}$, and $t_{i+1}>n_{2i+2}$ such that $u\in \text{span}\{I_j: t_i<j\leqslant t_{i+1}\}$.    Now for any $(a_i)_{i=1}^\infty \in c_{00}$, letting $a=\|(a_i)_{i=1}^\infty\|_{c_0}$, \begin{align*} \|\sum_{i=1}^\infty a_i x_{(n_1, \ldots, n_{2i})}\| & \leqslant a+\|\sum_{i=1}^\infty a_iu_i\| \leqslant a+C\|\sum_{i=1}^\infty a_i e_{m_{t_i}}\|_{\xi,p} \\ & \leqslant (1+CD)\|\sum_{i=1}^\infty a_i e_{n_{2i-1}}\|_{\xi,p} = C'\|\sum_{i=1}^\infty a_i e_{n_{2i-1}}\|_{\xi,p}. \end{align*}

\end{proof}

\begin{corollary} Suppose $X$ is a Banach space, $K\subset X^*$, $\textsf{\emph{F}}$ is a $K$-shrinking FMD for $X$, $0<\xi<\omega_1$, and $Sz(K)\leqslant \omega^\xi$.   Then there exist a blocking $\textsf{\emph{G}}$ of $\textsf{\emph{F}}$, a sequence $(m_n)_{n=1}^\infty \in [\nn]$,  and a constant $C\geqslant 0$ such that for each $0=r_0<r_1<\ldots$, each $(u_i)_{i=1}^\infty\in B_X^{<\nn}\cap \prod_{i=1}^\infty [G_j: r_{i-1}<j\leqslant r_i]$, and each $(a_i)_{i=1}^\infty\in c_{00}$, $$r_K(\sum_{i=1}^\infty a_i u_i) \leqslant C\|\sum_{i=1}^\infty a_i e_{m_{r_i}}\|_\xi.$$  

\label{cc}
\end{corollary} 

\begin{proof} Fix $R>\sup_{x^*\in K}\|x^*\|$.       As in Corollary \ref{nowitstimetoreallybaernnotice}, we recursively select $M_1\supset M_2\supset \ldots$ such that for all $n\in\nn$, $$\mathcal{H}(X, \textsf{F}, K, 2^m/2^n)\cap [M_n]^{<\nn}\subset \mathcal{S}_\xi.$$  We may do this, since $$CB(\mathcal{H}(X, \textsf{F}, K, 2^m/2^n))<\omega^\xi.$$    Now fix $m_1<m_2<\ldots$, $m_n\in M_n$ and let $G_n=[F_j: m_{n-1}<j\leqslant m_n]$, $0=r_0<r_1<\ldots$, $(u_i)_{i=1}^\infty \in B_X^\nn\cap \prod_{i=1}^\infty [G_j: r_{i-1}<j\leqslant r_i]$, and $(a_i)_{i=1}^\infty \in c_{00}$, $$r_K(\sum_{i=1}^\infty a_iu_i) \leqslant \|\sum_{i=1}^\infty a_ie_{m_{r_i}}\|_{X_\xi}\sum_{n=1}^\infty  \frac{2^m}{2^{n-1}}n.$$

\end{proof}

\section{Factorization and universality}

We first recall a construction of Schechtman.  There exists a sequence $\textsf{U}=(U_n)_{n=1}^\infty$ of finite dimensional spaces which form a bimonotone FDD for a Banach space $\mathfrak{U}$ such that if $X$ is any Banach space with bimonotone FDD $\textsf{F}=(F_n)_{n=1}^\infty$ and if $m_1<m_2<\ldots$ are natural numbers, then  there exist a sequence $k_1<k_2<\ldots$ of natural numbers, a sequence $I_n:F_n\to U_{k_n}$ of isomorphisms, and a projection $P:\mathfrak{U}\to [U_{k_i}: i\in\nn]$ such that \begin{enumerate}[(i)]\item $m_i<k_i$ for all $i\in\nn$, \item $\|I_n\|, \|I_n^{-1}\|\leqslant 2$, \item the map $x=\sum_{n=1}^\infty x_n\mapsto \sum_{n=1}^\infty I_n x_n$ defines an isomorphism  $I:X\to [U_{k_i}:i\in\nn]$ such that $\|I\|, \|I^{-1}\|\leqslant 2$, \item $\|P\|=1$. \end{enumerate} 

In the sequel, the symbol $\mathfrak{U}$ will be reserved for this space and the symbol $\textsf{U}$ will denote the FDD $(U_n)_{n=1}^\infty$ of $\mathfrak{U}$. 

\begin{proposition} \begin{enumerate}[(i)]\item If $P:\mathfrak{U}\to [U_{k_i}: i\in \nn]$ is the norm $1$ projection given by $P\sum_{i=1}^\infty x_i= \sum_{i=1}^\infty x_{k_i}$, then for any sequence space $E$, $P:\mathfrak{U}^E_\wedge(\textsf{U})\to \mathfrak{U}^E_\wedge(\textsf{U})$ is also norm $1$. \item If $E$ is a sequence space with property $R$, $k_1<k_2<\ldots$, $\mathfrak{V}=[U_{k_i}:i\in \nn]\subset \mathfrak{U}$, $V_n=U_{k_n}$, the projection $P:\mathfrak{U}\to \mathfrak{V}$ given by $P\sum_{i=1}^\infty x_i=\sum_{i=1}^\infty x_{k_i}$ is norm $1$,  $\textsf{\emph{V}}=(V_n)_{n=1}^\infty$, and the norms $\|\cdot\|_\mathfrak{U}$ and $\|\cdot\|_{\mathfrak{V}^E_\wedge(\textsf{\emph{V}})}$ are equivalent on $\mathfrak{V}$, then the norms $\|\cdot\|_\mathfrak{U}$ and $\|\cdot\|_{\mathfrak{U}^E_\wedge(\textsf{\emph{U}})}$ are equivalent on $\mathfrak{V}$. \item Suppose $\mathcal{G}$, $\mathcal{G}_0, \mathcal{G}_1, \ldots$ are regular families such that $CB(\mathcal{G})<\sup_n CB(\mathcal{G}_n)$ and  $1=\vartheta_0>\vartheta_1>\ldots$, $\lim_n \vartheta_n=0$. Let $E=X(\mathcal{G}_n, \vartheta_n)$ be the mixed Schreier space.   If $W$ is a Banach space with FDD $\textsf{\emph{F}}$ such that the norms $\|\cdot\|_W$ and $\|\cdot\|_{W^{X_\mathcal{G}}_\wedge(\textsf{\emph{F}})}$ are equivalent, then $W$ embeds complementably into $\mathfrak{U}^E_\wedge(\textsf{U})$.    \end{enumerate}

\label{Schechtman}
\end{proposition}

\begin{proof}$(i)$  For any $x\in c_{00}(\textsf{U})$, \begin{align*} [Px]_{\mathfrak{U}^E_\wedge(\textsf{U})} & = \inf\Bigl\{\|\sum_{i=1}^\infty \|I^\textsf{U}_i Px\|_\mathfrak{U}e_{\max I_i}\|_E:I_1<I_2<\ldots, \cup_{i=1}^\infty I_i=\nn\Bigr\} \\ & = \inf\Bigl\{\|\sum_{i=1}^\infty \|PI^\textsf{U}_i x\|_\mathfrak{U}e_{\max I_i}\|_E:I_1<I_2<\ldots, \cup_{i=1}^\infty I_i=\nn\Bigr\} \\ & \leqslant \inf\Bigl\{\|\sum_{i=1}^\infty \|I^\textsf{U}_i x\|_\mathfrak{U}e_{\max I_i}\|_E:I_1<I_2<\ldots, \cup_{i=1}^\infty I_i=\nn\Bigr\} \\ & = [x]_{\mathfrak{U}^E_\wedge(\textsf{U})} \end{align*}  and \begin{align*} \|Px\|_{\mathfrak{U}^E_\wedge(\textsf{U})} & = \inf\Bigl\{\sum_{i=1}^n [x_i]_{\mathfrak{U}^E_\wedge(\textsf{U})}: n\in\nn, Px=\sum_{i=1}^n x_i\Bigr\} \\ & \leqslant \inf\Bigl\{\sum_{i=1}^n [Px_i]_{\mathfrak{U}^E_\wedge(\textsf{U})}: n\in\nn, x=\sum_{i=1}^n x_i\Bigr\} \\ & \leqslant \inf\Bigl\{\sum_{i=1}^n [x_i]_{\mathfrak{U}^E_\wedge(\textsf{U})}: n\in\nn, x=\sum_{i=1}^n \Bigr\} \\ & = \|x\|_{\mathfrak{U}^E_\wedge(\textsf{U})} .\end{align*}

$(ii)$ Of course, $\|\cdot\|_{\mathfrak{U}^E_\wedge(\textsf{U})} \leqslant \|\cdot\|_\mathfrak{U}$. To establish the reverse inequality, it is sufficient to prove that  $\|x\|_{\mathfrak{V}^E_\wedge(\textsf{V})}\leqslant  \|x\|_{\mathfrak{U}^E_\wedge(\textsf{U})}$ for all $x\in c_{00}(\textsf{V})$.  To that end,  fix $x\in c_{00}(\textsf{V})$ and intervals $I_1<I_2<\ldots$ such that $\cup_{i=1}^\infty I_i=\nn$.   Let $J_i$ be such that for all $j\in \nn$, $J_j=\{i: k_i\in I_j\}$. Let $S=\{j: J_j\neq \varnothing\}$ and note that $(J_i)_{i\in S}$ are successive, $\nn=\cup_{i\in S}J_i$,  and $\max J_i\leqslant k_{\max J_i}\leqslant \max I_i$ for all $i\in S$.   Furthermore, \begin{align*} \|\sum_{i=1}^\infty \|I^\textsf{U}_i x\|_\mathfrak{U}e_{\max I_i}\|_E & = \|\sum_{i\in S} \|I^\textsf{U}_i x\|_\mathfrak{U} e_{\max I_i}\|_E = \|\sum_{i\in S} \|J^\textsf{V}_i \|_\mathfrak{U} e_{\max I_i}\|_E \geqslant \|\sum_{i\in S} \|J^\textsf{V}_i x\|_\mathfrak{U} e_{\max J_i}\|_E \\ & \geqslant [x]_{\mathfrak{V}^E_\wedge(\textsf{V})}.\end{align*}   From this it follows that $[x]_{\mathfrak{V}^E_\wedge(\textsf{V})} \leqslant [x]_{\mathfrak{U}^E_\wedge(\textsf{U})}$ for any $x\in c_{00}(\textsf{V})$.     Now for any $x\in c_{00}(\textsf{V})$, \begin{align*} \|x\|_{\mathfrak{U}^E_\wedge(\textsf{U})} & = \inf\Bigl\{\sum_{i=1}^n [x_i]_{\mathfrak{U}^E_\wedge(\textsf{U})}: x_i\in c_{00}(\textsf{U}), x=\sum_{i=1}^n x_i\Bigr\} \\ & \geqslant  \inf\Bigl\{\sum_{i=1}^n [Px_i]_{\mathfrak{U}^E_\wedge(\textsf{U})}: x_i\in c_{00}(\textsf{U}), x=\sum_{i=1}^n x_i\Bigr\} \\ & = \inf\Bigl\{\sum_{i=1}^n [x_i]_{\mathfrak{U}^E_\wedge(\textsf{U})}: x\in c_{00}(\textsf{V}), x=\sum_{i=1}^n x_i\Bigr\} \\ & \geqslant \inf\Bigl\{\sum_{i=1}^n [x_i]_{\mathfrak{V}^E_\wedge(\textsf{V})}: x\in c_{00}(\textsf{V}), x=\sum_{i=1}^n x_i\Bigr\} \\ & = \|x\|_{\mathfrak{V}^E_\wedge(\textsf{V})}. \end{align*}

$(iii)$ Fix $l\in\nn$ such that $CB(\mathcal{G})<CB(\mathcal{G}_l)$ and $M=(m_n)_{n=1}^\infty$ such that $\mathcal{G}(M)\subset \mathcal{G}_l$.   By renorming $W$, we may assume $\textsf{F}$ is bimonotone in $W$ and we may assume $W=W^{X_\mathcal{G}}_\wedge(\textsf{F})$.   Select $k_1<k_2<\ldots$, $I_n:F_n\to U_{k_n}$, and $P:\mathfrak{U}\to \mathfrak{V}=[U_{k_n}:n\in\nn]=[V_n:n\in\nn]$ satisfying (i)-(iv) as in the discussion of $\mathfrak{U}$.  Let us first note that for any $(a_i)_{i=1}^\infty\in c_{00}$, $$\|\sum_{i=1}^\infty a_i e_{k_i}\|_E \geqslant \vartheta_l \|\sum_{i=1}^\infty a_ie_i\|_{X_\mathcal{G}}.$$  Indeed, fix $F\in \mathcal{G}$ such that $$\|\sum_{i=1}^\infty a_i e_i\|_{X_\mathcal{G}} = \|F\sum_{i=1}^\infty a_ie_i\|_{\ell_1}.$$ Then $H:=\{k_i: i\in F\}$ is a spread of $M(F)$, and therefore lies in $\mathcal{G}_l$.  From this it follows that $$\|\sum_{i=1}^\infty a_i e_{k_i}\|_E \geqslant \vartheta_l \|H\sum_{i=1}^\infty a_ie_{k_i}\|_{\ell_1}= \vartheta_l \|F\sum_{i=1}^\infty a_ie_i\|_{\ell_1}= \vartheta_l\|\sum_{i=1}^\infty a_i e_i\|_{X_\mathcal{G}}.$$

We note that, since $I:W\to \mathfrak{V}$ is an isomorphism which takes $F_n$ to $V_n$, the norms $\|\cdot\|_\mathfrak{U}$ and $\|\cdot\|_{\mathfrak{V}^{X_\mathcal{G}}_\wedge(\textsf{V})}$ are equivalent on $\mathfrak{V}$.   Since $\|\cdot\|_{\mathfrak{U}^E_\wedge(\textsf{U})}\leqslant \|\cdot\|_\mathfrak{U}$, we know $\sum_{n=1}^\infty w_n\mapsto \sum_{n=1}^\infty I_nw_n$ extends to a bounded, linear map from $W$ into $\mathfrak{U}^E_\wedge(\textsf{U})$. In order to know this is an isomorphic embedding, it is sufficient to know that $$\vartheta_l \|x\|_{\mathfrak{V}^{X_\mathcal{G}}_\wedge(\textsf{V})} \leqslant \|x\|_{\mathfrak{U}^E_\wedge(\textsf{U})}$$ for all $x\in c_{00}(\textsf{V})$.   To that end, fix $x\in c_{00}(\textsf{V})$ and intervals $I_1<I_2<\ldots$ with $\cup_{i=1}^\infty I_i=\nn$.  Let $J_1, J_2, \ldots$ and $S$ be as in $(ii)$ and note that $k_{\max J_i}\leqslant \max I_i$ for all $i\in S$. Then \begin{align*} \|\sum_{i=1}^\infty \|I^\textsf{U}_i x\|_\mathfrak{U} e_{\max I_i}\|_E & = \|\sum_{i\in S} \|I^\textsf{U}_ix\|_\mathfrak{U} e_{\max I_i}\|_E \geqslant \|\sum_{i\in S} \|J^\textsf{V}_i x\|_\mathfrak{U} e_{k_{\max J_i}}\|_E \\ & \geqslant \vartheta_l \|\sum_{i\in S} \|J^\textsf{V}_ix\|_\mathfrak{U} e_{\max J_i}\|_{X_\mathcal{G}}.  \end{align*}   From this it follows that $$\vartheta_l [x]_{\mathfrak{V}^{X_\mathcal{G}}_\wedge(\textsf{V})} \leqslant [x]_{\mathfrak{U}^E_\wedge(\textsf{U})}$$ for all $x\in c_{00}(\textsf{V})$.  We now reach the desired conclusion as in $(ii)$, deducing that the image of $I$ is complemented in $\mathfrak{U}^E_\wedge(\textsf{U})$   by $(i)$.

\end{proof}

\begin{theorem} Fix $0<\xi<\omega_1$.  Let $X$ be a Banach space with shrinking FDD $\textsf{\emph{F}}$ and let $A:X\to Y$ be an operator with $Sz(A)=\omega^\xi$.  \begin{enumerate}[(i)]\item If $\xi=\omega^{\zeta+1}$, then $A$ factors through a Banach space $Z$ with Szlenk index $\omega^\xi$ if and only if there exists $\gamma<\omega^\zeta$ such that for each $n\in\nn$, $Sz(A, 2^{-n})\leqslant \gamma^n$. \item If $\xi=\omega^\zeta$, $\zeta$ a limit ordinal, $A$ does not factor through any Banach space $Z$ with $Sz(Z)=Sz(A)$. \item If $\xi=\beta+\gamma$ for some $\beta, \gamma<\xi$, then $A$ factors through a Banach space $Z$ with $Sz(A)=\omega^\xi$. \end{enumerate}

\end{theorem}

\begin{rem}\upshape If $X$ is a Banach space with bimonotone FDD $\textsf{G}$, $A:X\to Y$ is an operator,  $E$ is a sequence space, and $C\geqslant 1$ are such that for any $0=r_0<r_1<\ldots$ and any $(x_i)_{i=1}^\infty \in B_X^\nn\cap [G_j: r_{i-1}<j\leqslant r_i]$, then for any $(a_i)_{i=1}^\infty \in c_{00}$, $$\|A\sum_{i=1}^\infty a_ix_i\|\leqslant C\|\sum_{i=1}^\infty a_ie_{r_i}\|.$$   Then $A$ factors through $X_\wedge^E(\textsf{G})$.

Indeed, since $\|\cdot\|_{X^E_\wedge(\textsf{G})}\leqslant \|\cdot\|_X$, the formal inclusion $I:X\to X_\wedge^E(\textsf{G})$ is well-defined.   Fix $x\in c_{00}(\textsf{G})$ and suppose $$[x]_{X^E_\wedge(\textsf{G})}= \|\sum_{i=1}^\infty \|I_i^\textsf{G} x\|_X e_{\max I_i}\|_E.$$   Now if $I_i^\textsf{G}x\neq 0$, let $a_i=\|I^\textsf{G}_ix\|_X$ and let $x_i= a_i^{-1} I^\textsf{G}_i x$.    If $I_i^\textsf{G} x=0$, let $a_i=0$ and $x_i=0$. Then $x=\sum_{i=1}^\infty a_i x_i$ and, by hypothesis, $$\|Ax\|=\|A\sum_{i=1}^\infty a_ix_i\|\leqslant C \|\sum_{i=1}^\infty a_i e_{\max I_i}\|_E=C[x]_{X^E_\wedge(\textsf{G})}.$$   

Now for any $x\in c_{00}(\textsf{G})$, $$\|Ax\| \leqslant \inf\Bigl\{\sum_{i=1}^n \|Ax_i\|: n\in\nn, x=\sum_{i=1}^n x_i\Bigr\} \leqslant C \inf\Bigl\{\sum_{i=1}^n [x_i]_{X^E_\wedge(\textsf{G})}: n\in\nn, x=\sum_{i=1}^n x_i\Bigr\}=C\|x\|_{X^E_\wedge(\textsf{G})}.$$    From this it follows that $A|_{c_{00}(\textsf{G})}$ extends to a norm at most $C$ operator $J:X^E_\wedge(\textsf{G})\to Y$, and $A=JI$. 

\label{hops}

\end{rem}

\begin{proof}$(i)$ Suppose $\xi=\omega^{\zeta+1}$.   First suppose that $Z$ is a Banach space such that $A$ factors through $Z$ and $Sz(Z)=\omega^{\omega^{\zeta+1}}$.   Then there exists a constant $C\geqslant 1$ such that $Sz(A, \ee)\leqslant Sz(Z, \ee/C)$ for all $0<\ee<1$.  Since $Sz(Z)=\omega^{\omega^{\zeta+1}}$, there exists $\gamma<\omega^{\omega^{\zeta+1}}$ such that $Sz(Z, 1/2C)<\gamma$.   For all $n\in\nn$, $$Sz(A, 1/2^n) \leqslant Sz(Z, 1/C 2^n) \leqslant Sz(Z, 1/(2C)^n)\leqslant Sz(Z, 1/2C)^n<\gamma^n.$$

For the converse, suppose there exists $\gamma<\omega^{\omega^{\zeta+1}}$ such that $Sz(A, 1/2^n)<\gamma^n$ for all $n\in\nn$. Fix $m\in\nn$ such that $2^m>\|A\|$.     Fix a regular family $\mathcal{G}$ with $2\gamma<CB(\mathcal{G})<\omega^{\omega^{\zeta+1}}$.  Let $\mathcal{G}_0=\mathcal{S}_0$ and $\mathcal{G}_n=\mathcal{G}[\mathcal{G}_{n-1}]$ for $n\in\nn$.  Let $\vartheta=2/3$ and let $X=X(\mathcal{G}_n, \vartheta^n)$.   Note that $Sz(X)=\omega^{\omega^{\zeta+1}}$ by Lemma \ref{ww}.     Recursively select $\nn\supset M_1\supset M_2\supset \ldots$ such that for each $n\in\nn$, either $$\mathcal{H}(X, \textsf{F}, A^*B_{Y^*}, 1/2^{m+n})\cap [M_n]^{<\nn}\subset \mathcal{G}_n$$ or $$\mathcal{G}_n\cap [M_n]^{<\nn}\subset \mathcal{G}(X, \textsf{F}, A^*B_{Y^*}, 2^m/2^n).$$  Now since $$CB(\mathcal{H}(X, \textsf{F}, A^*B_{Y^*}, 2^m/2^n)) \leqslant 2 Sz(A, 1/2^n)< 2 \gamma^n \leqslant (2\gamma)^n < CB(\mathcal{G}_n),$$  it must be the case that $$\mathcal{H}(X, \textsf{F}, A^*B_{Y^*}, 2^m/2^n)\cap [M_n]^{<\nn}\subset \mathcal{G}_n$$ for all $n\in\nn$.     Fix $0=m_0<m_1<m_2<\ldots$ such that $m_n\in M_n$.   For each $n\in\nn$, let $G_n=[F_j: m_{n-1}<j\leqslant m_n]$ and let $\textsf{G}=(G_n)_{n=1}^\infty$.   Let $E$ be the sequence space whose norm is given by $\|\sum_{i=1}^\infty a_ie_i\|_E=\|\sum_{i=1}^\infty a_i e_{m_i}\|_X$.  Let $$C=2^{m+1}\sum_{n=1}^\infty \frac{n}{2^{n-1}}+ \frac{3^n}{4^n}<\infty.$$   

Suppose $0=r_0<r_1<\ldots$ and $(x_n)_{n=1}^\infty\in B_X^\nn\cap \prod_{n=1}^\infty [G_j:r_{n-1}<j\leqslant r_n]$.   Fix $(a_i)_{i=1}^\infty\in c_{00}$ and $y^*\in B_{Y^*}$ such that $$\|A\sum_{i=1}^\infty a_ix_i\|=A^*y^*(\sum_{i=1}^\infty a_i x_i).$$   For each $n\in\nn$, let $$B_n=\{i<n: |A^*y^*(x_i)|\in (2^m/2^n, 2^m/2^{n-1}]\}$$ and $$C_n=\{i\geqslant n: |A^*y^*(x_i)|\in (2^m/2^n, 2^m/2^{n-1}]\}.$$  Note that $|B_n|<n$, whence $$A^*y^*(\sum_{i\in B_n} a_i x_i) \leqslant \frac{2^m}{2^{n-1}}  \sum_{i\in B_n}|a_i|\leqslant \frac{n2^m}{2^{n-1}}\|\sum_{i\in B_n} e_{m_{r_i}}\|\leqslant \frac{n2^m }{2^{n-1}}\|\sum_{i\in B_n} e_{r_i}\|_E.$$ Note also that  $$(m_{r_i})_{i\in C_n}\in \mathcal{H}(X, \textsf{F}, A^*B_{Y^*}, 2^m/2^n)\cap [M_n]^{<\nn}\subset \mathcal{G}_n.$$  From this it follows that $$A^*y^*(\sum_{i\in C_n} a_i x_i) \leqslant \frac{2^m}{2^{n-1}}  \sum_{i\in C_n}|a_i|\leqslant \frac{2^m}{2^{n-1}}\cdot \frac{3^n}{2^n}\|\sum_{i\in C_n} e_{m_{r_i}}\| =  \frac{n2^{m+1} 3^n }{4^n}\|\sum_{i\in C_n} e_{r_i}\|_E.$$    Then $$\|A\sum_{i=1}^\infty a_ix_i\|=A^*y^*(\sum_{i=1}^\infty a_ix_i) \leqslant C\|\sum_{i=1}^\infty a_ie_{r_i}\|_E.$$   Then as noted in Remark \ref{hops}, $A$ factors through $X^E_\wedge(\textsf{G})$.   By Corollary \ref{7u}, $Sz(X^E_\wedge(\textsf{G}))=\omega^{\omega^{\zeta+1}}$.

$(ii)$ By \cite{C3}, there is no Banach space with Szlenk index $\omega^{\omega^\zeta}$, $\zeta$ a limit ordinal.

$(iii)$ Write $\xi=\beta+\gamma$ with $\beta, \gamma<\xi$.   Fix $m\in\nn$ such that $2^m>\|A\|$.    We may fix an increasing sequence $\gamma_n$ of ordinals such that $\gamma_n\uparrow \omega^\gamma$ and $2Sz(A, 2^m/2^{n+1})<\omega^\beta \gamma_n$.   Fix a sequence of regular families $\mathcal{F}_n$ with $\gamma_n<CB(\mathcal{F}_n)$ and let $\mathcal{G}_0=\mathcal{S}_\beta$, $\mathcal{G}_n=\mathcal{F}_n[\mathcal{G}_0]$.  Let $X=X(\mathcal{G}_n, (2/3)^n)$ be the mixed Schreier space and note that $X$ is $\xi$-well-constructed.  As in $(i)$, for each $n\in\nn$,  we find $ M_n$ such that $\mathcal{H}(X, \textsf{F}, A^*B_{Y^*}, 2^m/2^n)\cap [M_n]^{<\nn}\subset \mathcal{G}_n$.  We then select $m_1<m_2<\ldots $ such that $m_n\in M_n$.     Arguing as in $(i)$, with $\textsf{G}$, $E$,  $C$ defined in the same way, we deduce that $A$ factors through $X^E_\wedge(\textsf{F})$, which has Szlenk index $\omega^\xi$.

\end{proof}

\begin{rem}\upshape If $X$ has a shrinking FDD and $A:X\to Y$ is an operator with $Sz(A)=\omega$, then $A$ factors through a Banach space $Z$ with $Sz(Z)=\omega$ if and only if there exists $l\in\nn$ such that $Sz(A,1/2^n) \leqslant l^n$ for all $n\in\nn$.  This has already been shown in \cite{KP}.  Our argument above is essentially a transfinite extension of this fact.  Indeed, in this case, $\textsf{p}_0(A)<\infty$.  Therefore if $1/r+1/s=1$ and $\textsf{p}_0(A)<s<\infty$, $A$ factors through $X^{X_{0,r}}_\wedge(\textsf{G})=X^{\ell_r}_\wedge(\textsf{G})$ as a consequence of Corollary \ref{nowitstimetoreallybaernnotice} and Remark \ref{hops}.

\end{rem}

\begin{theorem} For any $\xi\in (0,\omega_1)\setminus \{\omega^\eta: \eta\text{\ a limit ordinal}\}$, there exists a Banach space $\mathfrak{G}_\xi$ with a shrinking basis and $Sz(\mathfrak{G}_\xi)=\omega^\xi$ such that if $A:X\to Y$ is a separable range operator with $Sz(A)<\omega^\xi$, then $A$ factors through a subspace and through a quotient of $\mathfrak{G}_\xi$.   Moreover, if $X$ has a shrinking FDD, $A$ may be taken to factor through $\mathfrak{G}_\xi$.  

\label{universal theorem}

\end{theorem}

\begin{proof} Case 1: $\xi$ is a  successor, say $\xi=\zeta+1$.    Let $\mathfrak{S}_\xi=\mathfrak{U}^{X_{\zeta, 2}}_\wedge(\textsf{U})$.   By a technique of Pe\l czy\'{n}ski, for each $n\in \nn$, there exists a finite dimensional space $I_n$ having  basis with basis constant not more than $2$ and such that $U_n\leqslant I_n$ and $U_n$ is $2$-complemented in $I_n$.  Let $P_n:I_n\to U_n$ be a projection with norm  not more than $2$ and let $J_n=\ker(P_n)$.   Let $\mathfrak{G}_\xi=\mathfrak{S}_\xi\oplus_\infty (\oplus_{n=1}^\infty J_n)_{c_0}$.     Then $\mathfrak{G}_\xi$ has a shrinking basis and $$Sz(\mathfrak{G}_\xi)=\max\{Sz(\mathfrak{S}_\xi), Sz(\oplus_{n=1}^\infty J_n)_{c_0})\}=\omega^{\zeta+1}.$$   

Now suppose that $A:X\to Y$ is an operator with separable range and  $Sz(A)\leqslant \omega^\zeta$.    Then by \cite{TAMS}, $A$ factors through a separable Banach space $Z$ with $\textsf{p}_\zeta(Z)<2$.      By Theorem \ref{ute}, there exists a Banach space $W$ with FDD $\textsf{F}$ such that $Z$ is isomorphic to both a subspace and to a quotient of $W^{X_{\zeta, 2}}_\wedge(\textsf{F})$.      By Proposition \ref{Schechtman}, $W^{X_{\zeta, 2}}_\wedge(\textsf{F})$ is isomorphic to a complemented subspace of $\mathfrak{S}_\xi$, and is therefore isomorphic to a complemented subspace of $\mathfrak{G}_\xi$.  Then $A$ factors through $Z$, which is isomorphic to both a subspace and to a quotient of $\mathfrak{G}_\xi$.

If $X$  has a shrinking FDD, say $\textsf{F}$, then by Corollary \ref{nowitstimetoreallybaernnotice} and Remark \ref{hops}, there exists a blocking $\textsf{G}$ of $\textsf{F}$ such that $A$ factors through $X^{X_{\zeta, 2}}_\wedge(\textsf{G})$, since $\textsf{p}_\zeta(A)=0$.   The space $X^{X_{\zeta, 2}}_\wedge(\textsf{G})$ is isomorphic to a complemented subspace of $\mathfrak{G}_\xi$, whence $A$ can be taken to factor through $\mathfrak{G}_\xi$.

Case $2$: $\xi$ is a limit ordinal. Let $E=X(\mathcal{G}_n, \vartheta_n)$ be a $\xi$-well-constructed  mixed Schreier space, so that  $Sz(E)=\omega^\xi$, $\mathfrak{S}_\xi=\mathfrak{U}^E_\wedge(\textsf{U})$, and $\mathfrak{G}_\xi=\mathfrak{S}_\xi\oplus_\infty(\oplus_{n=1}^\infty J_n)_{c_0}$, where $J_n$ is chosen as in the previous case. Then $Sz(\mathfrak{G}_\xi)=\omega^\xi$.  

Now suppose that $A:X\to Y$ is a separable range operator with $Sz(A)=\omega^\zeta<\omega^\xi$.     Then $A$ factors through a separable Banach space $Z$ with $Sz(Z)\leqslant \omega^{\zeta+1}$ by \cite{Brooker}, and $Z$ is isomorphic to both a subspace and to a quotient of a Banach space $W^{X_{\zeta+1}}_\wedge(\textsf{F})$, where $W$ is some Banach space and $\textsf{F}$ is an FDD for $W$ by Theorem \ref{ute}.   Then by Proposition \ref{Schechtman}, $W^{X_{\zeta+1}}_\wedge(\textsf{F})$ is isomorphic to a complemented subspace of $\mathfrak{G}_\xi$.

If $X$ has a shrinking FDD, say $\textsf{F}$, then by  Corollary \ref{cc} and Remark \ref{hops}, there exists a blocking $\textsf{G}$ of $\textsf{F}$ such that $A$ factors through $X^{X_{\zeta+1}}_\wedge(\textsf{G})$.    By Proposition \ref{Schechtman}, $X^{X_{\zeta+1}}_\wedge(\textsf{G})$ embeds complementably in $\mathfrak{G}_\xi$, and so $A$ can be taken to factor through $\mathfrak{G}_\xi$ in this case.

\end{proof}

\begin{rem}\upshape By a result of Johnson and Szankowski \cite{JS}, there is no separable Banach space through which all compact operators factor.  A consequence of this result is that for $0<\xi<\omega_1$, there cannot be a separable Banach space such that every operator with Szlenk index less than $\omega^\xi$ factors through that space, since every compact operator has Szlenk index $1$. Thus having the operators in Theorem \ref{universal theorem} factor through a subspace or a quotient rather than through the whole space is necessary.

\end{rem}

\begin{rem}\upshape  For any ordinal $\xi<\omega_1$ and any Banach space $Z$ with $Sz(Z)=\omega^\xi$, there is a separable Banach space $X$ (which can be taken to be a mixed Schreier space if $0<\xi$) with $Sz(Z)=Sz(X)$ such that $X$ is not isomorphic to any subspace of any quotient of $Z$.    Indeed, if $\xi=0$ and $Sz(Z)=\omega^\xi=1$, $\dim Z<\infty$, and we simply take $X$ to have $\dim Z<\dim X<\infty$. 

  If $Z=\omega^{\zeta+1}$, we fix $Sz(Z, 1/2)<\gamma<\omega^{\omega^{\zeta+1}}$ and let $X$ be a $\xi$-well-constructed mixed Schreier space with $Sz(X)=\omega^{\omega^{\zeta+1}}$ such that $Sz(X, (2/3)^n)\geqslant \gamma^n$.    If $\xi=\beta+\gamma$ for $\beta, \gamma<\xi$, we fix $\gamma_n<\omega^\gamma$ such that $Sz(Z, (1/2)^n) < \omega^\beta \gamma_n$ and construct a mixed Schreier space $X$ such that $Sz(X, (2/3)^n)>\omega^\beta \gamma_n$ and $Sz(X)=\omega^\xi$.  In either of these two cases, $X$ cannot be isomorphic to a subspace of a quotient of $Z$, otherwise there would exist some $C\geqslant 1$ such that $$Sz(X, \ee)\leqslant Sz(Z, \ee/C)$$ for all $0<\ee<1$.  But if $n\in\nn$ is such that $(1/2)^n<(2/3)^n/C$, our choice of $X$ yields that $$\gamma^n \leqslant Sz(X, (2/3)^n) \leqslant Sz(Z, (2/3)^n/C) \leqslant Sz(Z, (1/2)^n)<\gamma^n$$ in the first case, and $$\omega^\beta \gamma_n \leqslant Sz(X, (2/3)^n) \leqslant Sz(Z, (2/3)^n/C) \leqslant Sz(Z, (1/2)^n)<\omega^\beta \gamma_n$$ in the second case. 

\end{rem}

\begin{rem}\upshape For $\xi=\omega^\zeta$, $\zeta$ a limit ordinal, there is no Banach space with Szlenk index $\omega^\xi$, which is the reason Theorem \ref{universal theorem} is limited to $\xi\in (0, \omega_1)\setminus \{\omega^\eta: \eta\text{\ a limit ordinal}\}$.  However, for any countable limit ordinal $\eta$, we may fix a sequence $\eta_n\uparrow \eta$ and define $\mathfrak{H}_\eta=(\oplus_{n=1}^\infty \mathfrak{G}_{\omega^{\eta_n}+1})_{c_0}$.  We may define a diagonal operator $A_\eta:\mathfrak{H}_\eta\to \mathfrak{H}_\eta$ by $A_\eta|_{\mathfrak{G}_{\omega^{\eta_n}+1}}=\frac{1}{n}I_{\mathfrak{G}_{\omega^{\eta_n}+1}}$.   Then $$Sz(A_\eta)=\sup_n Sz(\mathfrak{G}_{\omega^{\eta_n}+1})=\sup_n \omega^{\omega^{\eta_n}+1}=\omega^{\omega^\eta},$$ and if $A:X\to Y$ is any operator between separable Banach spaces with $Sz(A)<\omega^{\omega^\eta}$, there exist a subspace $Z$ of $\mathfrak{H}_\eta$ and operators $R:X\to Z$, $L:Z\to Y$ such that $A_\eta(Z)=Z$ and  such that $LA_\eta R=A$.

\end{rem}

\end{document}